\documentclass[twoside]{article}

\usepackage{star}
\usepackage[accepted]{aistats2022}
\usepackage{mathtools}
\usepackage{hyperref}
\hypersetup{
    colorlinks=true,
    linkcolor=blue,
    filecolor=blue,
    urlcolor=blue,
    citecolor=blue,
}
\usepackage[round]{natbib}

\begin{document}

%

%

\twocolumn[

\aistatstitle{Distributed Sparse Multicategory Discriminant Analysis}

\aistatsauthor{ Hengchao Chen \And Qiang Sun }

\aistatsaddress{ University of Toronto \And  University of Toronto } ]

\begin{abstract}
  This paper proposes a convex formulation for sparse multicategory linear discriminant analysis  and then extend it to the distributed setting when data are  stored across multiple sites. The key observation  is that for the purpose of classification it suffices to recover the discriminant subspace which is invariant to orthogonal transformations. Theoretically, we establish statistical properties ensuring that the distributed sparse multicategory linear discriminant analysis performs as good as the centralized version after {a few rounds} of communications. Numerical studies lend strong support to our methodology and theory.
\end{abstract}

\section{INTRODUCTION}
Classification aims to assign data points to the correct classes. One popular multicategory classification method is the Fisher's linear discriminant analysis (LDA), which has appealing performances in many applications \citep{Hand06,Michie94}. Similar to the principal component analysis (PCA), a well-known unsupervised dimensionality reduction technique, LDA can also serve as an efficient supervised dimensionality reduction tool. However, when the number of features is larger than the number of observations, which prevails in  modern datasets, LDA performs poorly due to the diverging spectra \citep{Bickel08} and the noise accumulation \citep{Fan08}. In the moderately high dimensional setting, the naive Bayes, which assumes conditional independence between features, is shown to outperform the Fisher's rule \citep{Bickel04}, while in the ultra-high dimensionsal setting, any classifier using all the predictors will be no better than random guessing due to the noise accumulation, even if the true covariance matrix is an identity matrix \citep{Fan08}. Moreover, as pointed out by \citet{Witten11}, the Fisher's LDA classifier is not interpretable when the discriminant vectors  have no particular structure.

In recent years, many high-dimensional extensions of LDA have been proposed. For binary classification problems, the linear programming discriminant \citep{Cai11}, the regularized optimal affine discriminant \citep{Fan12}, and the direct sparse discriminant analysis \citep{Mai12} are three popular sparse linear discriminant analysis (SLDA) methods. All these three methods assume some sparsity assumptions on the discriminant directions. Another SLDA method is the thresholding linear discriminant analysis \citep{Shao11}, which assumes sparsity conditions on the common covariance matrix and the mean difference vector. It is non-trivial to extend these binary SLDA methods to the multicategory case. For multicategory classification problems, \citet{Qiao09} proposed to solve a penalized least squares problem with the lasso penalty, \citet{Clemmensen11} proposed the sparse optimal scoring, and \citet{Witten11} proposed the $\ell_1$ penalized LDA. All of these methods involve solving a non-convex optimization problem and thus can be computationally expensive. Moreover, it is unclear how to extend these methods to a distributed setting when data are possibly stored in multiple sites. Recently, \citet{Mai19} proposed to directly estimate the Bayes rule via a composite-type loss function. \citet{Safo16} proposed to estimate the discriminant vectors by applying a basis approach to the methods studied in \citet{Cai11} and \citet{Shao11}. \citet{Gaynanova16} proposed to extract the sparse discriminant vectors simultaneously through a convex programming problem. The underlying idea of these three methods are the same: for the purpose of classification, it suffices to accurately estimate the discriminant subspace, i.e., the subspace spanned by the discriminant vectors. In this paper, we will exploit this observation and propose an $\ell_{1,1}$ regularized multicategory linear discriminant annalysis method, which can also  estimate the sparse discriminant vectors simultaneously. It is worth mentioning that we choose the $\ell_{1,1}$ penalty instead of the $\ell_{2,1}$ penalty  in \citet{Mai19} and \citet{Gaynanova16}, so that we can easily extend the multicategory LDA method to a distributed setting and establish the corresponding statistical properties.

In addition to challenges posed by the high dimensionality of many datasets,  with the rapid developments of science and technology, we have seen more and more  datasets that are  often scattered across distant servers, possibly due to the limitation of storage resources. The  difficulties of fusing or aggregating these datasets due to the communication cost and privacy concerns have inspired many works in communication-efficient statistical learning \citep{Lee17,Wang17,Fan19,Jordan19}. One popular distributed estimation framework is the one-shot divide-and-conquer algorithm, in which one first computes the local estimators and then obtains the distributed estimator as an average of the local ones \citep{Lee17,Fan19}. These averaging-based approaches suffer from at least three drawbacks: non-diminishing bias when  local estimators are biased,  allowing only a limited number of local machines (much smaller than $\sqrt{N}$, where $N$ is the total sample size), and poor performance in nonlinear problems \citep{Jordan19}. To overcome these issues, \citet{Jordan19} and \citet{Wang17} studied a multi-round framework for distributed learning, which is referred to as the communication-efficient surrogate likelihood (CSL) framework by \citet{Jordan19}. 

Although distributed classification has attracted some attention \citep{Kokiopoulou10,Wang19,Lian18}, few research has been done on the distributed sparse discriminant analysis. For binary classification, \citet{Tian17} proposed a divide-and-conquer sparse LDA method, in which they  averaged the debiased local estimators of the discriminant direction and then sparsified the aggregated estimator. Moreover, they showed that their  distributed sparse LDA method can achieve the same performance as the centralized one which used all samples. However, it is not straightforward to generalize their approach to the multicategory case. To handle the multicategory case, in this paper, we propose an $\ell_{1,1}$ regularized multicategory SLDA formulation and then extend it to the distributed setting via the CSL framework. Theoretically, we show that the distributed multicategory SLDA (dmSLDA) method performs as good as the centralized SLDA after a few rounds of communications. Moreover, we conduct numerical experiments to further support our methodology and theory.

The rest of this paper proceeds as follows. In Section \ref{sec:2}, we review the multicategory sparse discriminant analysis and propose an $\ell_{1,1}$ regularized multicategory SLDA method. We then extend it to the distributed setting and propose the dmSLDA method in Section \ref{sec:3}. Section \ref{sec:4} establishes the $\ell_{2,2}$ and $\ell_{1,1}$ estimation error bounds of dmSLDA under certain conditions and explain how the proposed distributed estimator achieves the same performance as the centralized version. To further back up the methodology and theory, we conduct several numerical experiments in Section \ref{sec:5}. In Section \ref{sec:6}, we conclude this paper with some discussions. 

\noindent{\bf Notation.} We summarize some notations that will be used throughout this paper. By convention, we use regular letters for scalars and bold letters for both vectors and matrices. We employ $[M]$ to represent the set $\{1,\ldots,M\}$ for any positive integer $M$ and $[(a,b)]$ to abbreviate the set $\{(1,1),(1,2),\ldots,(a,b)\}$ for any positive integers $a$ and $b$. For a vector $\bu=(u_1,\ldots,u_d)^\top \in \RR^d$, we denote its $\ell_q$ norm by $\norm{\bu}_q=(\sum_{i=1}^d|u_i|^q)^{1/q}$, its $\ell_0$ norm by $\norm{\bu}_0 = \sum_{i=1}^d 1_{u_i\neq 0}$, and its $\ell_\infty$ norm by $\norm{\bu}_{\infty}=\max_{1\leq i\leq d}|u_i|$.  Moreover, for any two vectors $\bu,\bv\in\RR^d$, we use $\inner{\bu}{\bv}=\bu^\top \bv$ to denote the Euclidean inner product of $\bu$ and $\bv$. For a matrix $\bA\in\RR^{m\times n}$, we define the $\ell_{p,q}$ norm of $\bA$ as $\opnorm{\bA}{p,q}= \norm{\rbr{\norm{\bA_{1\cdot}}_p, \ldots, \norm{\bA_{m\cdot}}_{p}}}_q,$ where $\bA_{i\cdot}$ is the $i$th row of $\bA$. For any two matrices $\bA,\bB\in\RR^{m\times n}$, we denote $\inner{\bA}{\bB}=\textnormal{Tr}(\bA^\top\bB)$. For a matrix $\bA\in\RR^{m\times n}$, we use $\textnormal{span}(\bA)$ to denote the subspace spanned by the columns of $\bA$. For two sequences of real numbers $\cbr{a_n}_{n\geq 1}$ and $\cbr{b_n}_{n\geq 1}$, we write $a_n \lesssim b_n$ if $a_n \leq C b_n$ for some constant $C>0$ independent of $n$.

\section{MULTICATEGORY SLDA}\label{sec:2}
Let $\cbr{(\bx_i,y_i)}_{i=1}^N$ be a collection of $N$ $\textnormal{i.i.d.}$ samples with $y_i\in[K]$ drawn from a discrete distribution $\PP(y_i=k)=\pi_k$ and $\bx_i\in\RR^d$ drawn from a distribution satisfying $\EE[\bx_i|y_i=k]=\bmu_k$ and $\Var(\bx_i|y_i=k)=\bSigma$, where $K$ is the number of classes, $\pi_k$ is the prior probability of class $k$ such that $\pi_k\geq 0$ and $\sum_{k=1}^K\pi_k=1$, and $\bmu_k$ and $\bSigma$ are the class-conditional mean and covariance matrix of $\bx_i$ in class $k$, respectively. The $K$ centroids lie in an affine space of dimension $r_0\leq K-1$, which can be much smaller than $d$ \citep{Hastie09}. To exploit the underlying low dimensional structure, Fisher proposed to solve a set of discriminant vectors $\{\bv_k^*\}_{k=1}^q\subset\RR^d$ by maximizing the following sequence of Rayleigh quotients under orthogonal constraints,
\#
&\bv_k^*=\argmax_{\bv_k\in\RR^d}\frac{\bv_k^{\top}\bB\bv_k}{\bv_k^{\top}\bSigma\bv_k},\label{equ:2.1}\\
&\textnormal{s.t.}\quad \bv_k^{*\top}\bSigma\bv_j^*=0,\quad\forall 1\leq j\leq k\leq q,\notag
\#
where $\bB=\Var\rbr{\EE[\bx|y]}$ is the between-class covariance and $\bSigma=\EE[\Var\rbr{\bx|y}]$ is the within-class covariance. In other words, Fisher's proposal seeks a low dimensional projection of samples that maximizes the between-class variation relative to the within-class variation. Classification can then be performed on the projected data $\bV_q^\top\bx\in\RR^q$, where $\bV_q=(\bv_1^*,\ldots,\bv_q^*)\in\RR^{d\times q}$. In the classical setting where $N\gg d$, the unknown between-class covariance $\bB$ and within-class covariance $\bSigma$ in \eqref{equ:2.1} are substituted in practice by their corresponding sample versions $\hat\bB$ and $\hat\bSigma$ respectively, where
\#
\hat\bB & = \sum_{k=1}^{K}\hat\pi_k(\hat\bmu_k - \hat{\bar\bmu})(\hat\bmu_k - \hat{\bar\bmu})^{\top},\notag\\
\hat \bSigma & = \frac{1}{N-K}\sum_{k=1}^K\sum_{y_i=k}(\bx_i-\hat\bmu_k)(\bx_i-\hat\bmu_k)^{\top},\label{equ:2.2}
\#
in which $\hat\bmu_k=\frac{1}{b_k}\sum_{y_i=k}\bx_i$ is the sample mean of  class $k$, $b_k=\sum_{i=1}^N 1_{y_i=k}$ is the sample size of class $k$, $\hat\pi_k=b_k/N$ is the sample proportion of class $k$, and  $\hat{\bar\bmu}=\sum_{k=1}^{K}\hat\pi_k\hat\bmu_k=\frac{1}{N}\sum_{i}\bx_i$ is the overall sample mean.

In high dimensions when $d>N$, the Fisher's approach \eqref{equ:2.1} suffers from the singularity of $\hat\bSigma$, the noise accumulation, and the lack of interpretability \citep{Bickel04,Fan08,Witten11}. To overcome these issues, \citet{Witten11} and \citet{Clemmensen11} proposed sparse discriminant analysis, which requires solving a sequence of non-convex optimization problems and thus can be computationally intractable. On the other hand, \citet{Mai19}, \citet{Safo16} and \citet{Gaynanova16} proposed to compute any set of basis vectors of the discriminant subspace, i.e., $\cV_q\coloneqq \textnormal{span}(\bV_q)$ and then sparsified these basis vectors. Though not directly reported,  the core of these ideas is the following proposition.

\begin{proposition}\label{proposition:1}
	\textit{Suppose $\bV,\bW\in\RR^{d\times q}$ span the same subspace, then the LDA classification rule based on $\bV^\top\bx$ is equivalent to the LDA classification rule based on $\bW^\top\bx$.}
\end{proposition}

Proposition \ref{proposition:1} is a generalization of Proposition 4 in \citet{Gaynanova16}, implying that for the purpose of classification, it suffices to find any set of basis vectors of the discriminant subspace. 

\begin{proposition}\label{proposition:2}
	\textit{Assume that the within-class covariance $\bSigma$ is nonsingular and define $\bW^*\in\RR^{d\times(K-1)}$ as the solution to the following optimization problem
		\#
		\bW^*=\argmin_{\bW\in\RR^{d\times(K-1)}}\cbr{\frac{1}{2}\inner{\bW}{\bSigma \bW} - \inner{\bW}{\bU}},\label{equ:2.3}
		\#
		where $\bU=(\bmu_1-\bar\bmu,\ldots,\bmu_{K-1}-\bar\bmu)\in\RR^{d\times (K-1)}$ and $\bar\bmu = \sum_{k=1}^K \pi_k \bmu_k$. Then we have $\cW\coloneqq\textnormal{span}(\bW^*)=\cV_{\textnormal{rank}(\bB)}$, where $\bB$ is the between-class covariance.}
\end{proposition}

Proposition \ref{proposition:2} gives an explicit choice of the basis vectors of the discriminant subspace, simpler than that in \citet{Mai19} or \citet{Gaynanova16}. Since $\bSigma$ and $\bU$ in \eqref{equ:2.3} are unknown in practice, we will substitute them by their corresponding sample versions $\hat\bSigma$ given by \eqref{equ:2.2} and $\hat\bU=(\hat\bmu_1-\hat{\bar\bmu},\ldots,\hat\bmu_{K-1}-\hat{\bar\bmu})$, respectively. In high dimensions, we assume element-wise sparsity on $\bW^*$ and propose the following $\ell_{1,1}$ regularized estimator
\#
\hat\bW=\argmin_{\bW\in\RR^{d\times (K-1)}}\Big\{&\frac{1}{2}\langle\bW,\hat\bSigma\bW\rangle-\langle\bW,\hat\bU\rangle\notag\\
& +\lambda\opnorm{\bW}{1,1}\Big\},\label{equ:2.4}
\#
where $\lambda>0$ is a tuning parameter. We choose the $\ell_{1,1}$ penalty (or the element-wise sparsity) rather than the $\ell_{2,1}$ penalty (or the row sparsity) used in \citet{Mai19} and \citet{Gaynanova16} because only with the $\ell_{1,1}$ penalty shall we extend the multicategory SLDA method to the distributed setting easily. Moreover, the assumption of element-wise sparsity is a weaker assumption than the row sparsity. Finally, to perform  classification, one can apply LDA to the projected data $\hat\bW^\top\bx\in\RR^{K-1}$.

\begin{remark}
	\textit{The optimization problem \eqref{equ:2.4} and similarly \eqref{equ:3.1} and \eqref{equ:3.2} in the next section can be solved efficiently using the fast iterative shrinkage-thresholding algorithm (FISTA) proposed by \citet{Beck09}.}
\end{remark}

\section{DISTRIBUTED MULTICATEGORY SLDA}\label{sec:3}
In this section, we extend the proposed $\ell_{1,1}$ regularized multicategory SLDA method to the distributed setting. Suppose there are total of $M$ machines and on machine $m\in[M]$, $n$ samples\footnote{Here we assume for simplicity that the sample sizes across different machines are equal. In general, our framework can allow different sample sizes on each machine as long as the sample sizes are approximately balanced; that is the proportion of any single class can not vanish to zero.} 
$\cbr{(\bx^m_i,y^m_i)}_{i=1}^{n}$ are drawn independently from the distribution specified in Section \ref{sec:2}. Moreover, samples on different machines are assumed to be independent. To recover the discriminant subspace, the most straightforward yet unrealistic approach is to aggregate data from all machines together and then apply the $\ell_{1,1}$ regularized multicategory SLDA method based on the full sample. Another approach is to gather the local statistics $\hat\bSigma^m$ and $\hat\bU^m$ from all machines and then estimate $\bW^*$ using  \eqref{equ:2.4} with the aggregated statistics $\hat\bSigma=\frac{1}{M}\sum_{m=1}^M\hat\bSigma^m$ and $\hat\bU=\frac{1}{M}\sum_{m=1}^M\hat\bU^m$. Although these two centralized methods usually give the best estimation and prediction performance, they are communication expensive, with communication cost $\cO(ndM)$ and $\cO(d^2M)$, respectively. This motivates us to design a communication-efficient algorithm that achieves the same performance as the centralized SLDA.

We apply the communication-efficient surrogate likelihood framework \citep{Jordan19} to the multicategory SLDA method \eqref{equ:2.4} and  obtain the distributed multicategory SLDA method (dmSLDA). The full algorithm is collected in Algorithm \ref{algorithm:1}.
\begin{algorithm}[htbp]\label{algorithm:1}	
	\caption{Distributed multicategory SLDA (dmSLDA).}
	\LinesNumbered
	\KwIn{$\cbr{(\bx_i^m,y_i^m)}_{i=1}^{n_m}$, $T$}
	\KwOut{$\hat \bW$}
	compute $\hat\bSigma^m$ and $\hat \bU^m$ for all $m$\;	
	
	compute $\hat \bW_0$ as the solution to \eqref{equ:3.1}\;
	
	\For{$t$ \textnormal{in} $1:T$}{
		send $\hat \bW_{t-1}$ to local machines\;
		
		compute $\nabla\cL_m(\hat\bW_{t-1})$ on local machines\;
		
		send $\nabla\cL_m(\hat\bW_{t-1})$ back to the master machine\;
		
		compute $\hat\bW_{t}$ as the solution to \eqref{equ:3.2}\;
	}
	return $\hat\bW=\hat\bW_{t^*}$ with $t^*=\argmin_{0\leq t\leq T}\hat e_t$\;
\end{algorithm}
Specifically,  we first compute the local estimators $\hat\bSigma^m$ and $\hat\bU^m$. Then we compute the initial estimator $\hat\bW_0$ on the master machine, which is set to be the first local machine without loss of generality, as the solution to the following $\ell_{1,1}$ regularized optimization problem,
\#\label{equ:3.1}
\hat \bW_0 = \argmin_{\bW\in\RR^{d\times (K-1)}}\cbr{\cL_1(\bW)+\lambda_0\opnorm{\bW}{1,1}},
\#
where $\lambda_0>0$ is a tuning parameter and $\cL_m(\bW)$ is the loss function on machine $m$ given by
\$
\cL_m(\bW)=\frac{1}{2}\langle\bW,\hat\bSigma^m\bW\rangle-\langle\bW,\hat\bU^m\rangle.
\$ 
At the $t$th iteration, we send the $(t-1)$th estimator $\hat\bW_{t-1}$ to local machines, compute the gradients $\nabla\cL_m(\hat\bW_{t-1})$ of the $m$th loss function $\cL_m(\cdot)$ at $\hat\bW_{t-1}$, and then send these gradients back to the master machine. Then on the master machine, we compute the $t$th estimator by minimizing the following shifted $\ell_{1,1}$ regularized objective function,
\#
&\hat \bW_{t} = \argmin_{\bW\in\RR^{d\times (K-1)}}\Big\{\cL_1(\bW)+\lambda_{t}\opnorm{\bW}{1,1}\notag\\
&+\langle\frac{1}{M}\sum_{m=1}^M\nabla\cL_m(\hat\bW_{t-1})-\nabla\cL_1(\hat\bW_{t-1}),\bW\rangle\Big\},\label{equ:3.2}
\#
where $\lambda_t>0$ is a tuning parameter. After repeating the procedure for $T$ times, the dmSLDA algorithm returns $\hat \bW=\hat\bW_{t^*}$ with $t^*=\argmin_{0\leq t\leq T}\hat e_t$, where $\hat e_t=\sum_{m=2}^M\cL_m(\hat \bW_t)$. In other words, we choose the best estimator in $\{\hat\bW_t\}_{t=0}^T$ in the sense that it minimizes the validation loss. Finally, one can perform classification based on the projected data $\hat\bW^{\top}\bx\in\RR^{K-1}$. 

\begin{remark}
	\textit{The communication cost of the dmSLDA method is $\cO(dKM)$ since it only communicates the estimators $\hat\bW_{t}$ and the gradients $\nabla\cL_m(\hat\bW_{t-1})$. This is significantly smaller than that of the centralized methods, i.e., $\cO(ndM)$ (communicate all raw data $\{(\bx_{i}^m,y_i^m)\}$) or $\cO(d^2M)$ (communicate local  estimators $\hat\bSigma^m$ and $\hat\bU^m$), when $K\ll d,n$.}
\end{remark}

\begin{remark}
	\textit{In numerical experiments, we propose to choose $\lambda_t$ from {a candidate set $\Lambda_t$} by minimizing the validation loss $\hat e(\hat\bW_t)\coloneqq\sum_{m=2}^M\cL_m(\hat\bW_t)$, i.e., the loss on all other machines.}
\end{remark}

\begin{remark}
	\textit{Since in general a few iterations suffice for the procedure to match the same accuracy of the centralized estimator \citep{Jordan19}, we  choose $T=3$ or $5$ in practice.}
\end{remark}

\section{STATISTICAL ANALYSIS}\label{sec:4}
In this section, we establish statistical properties ensuring that dmSLDA can achieve the performance of the centralized SLDA after a few rounds of communication. We prove this in two steps: first, we will establish a recursive estimation error bound, i.e., we upper bound the estimation error $\opnorm{\hat\bW_{t+1}-\bW^*}{}$ in terms of the error in the previous round $\opnorm{\hat \bW_t-\bW^*}{}$; second, by applying this recursive estimation error bound iteratively, we can obtain an estimation error bound indicating that dmSLDA matches the performance of the centralized SLDA after a few rounds. 

To begin with, we make a few assumptions. For the sake of simplicity, we consider a balanced setting defined below\footnote{If the balance assumption does not hold and some class has significantly small prior probability and subsample size, one may refer to literature related for imbalanced data, such as \citet{Krawczyk16} and the references therein.}. In the balanced setting, the prior probability and the subsample size of each class are equal across different classes and machines. Our framework does allow approximately balanced settings when the sample size proportions of any single class and any single machine do not vanish. 

\begin{assumption}[Balanced]\label{assumption:1}
	\textit{We assume all classes have equal prior probability, i.e., $\pi_k=\frac{1}{K}$. Moreover, the subsample size of class $k$ on machine $m$ is assumed to be constant, denoted by $b$. Thus, the subsample size on a single machine is $n=Kb$ and the full sample size is $N=Mn$.} 
\end{assumption}

Furthermore, we assume that the conditional distribution of $\bx$ given $y=k$ is sub-Gaussian with mean $\bmu_k$ and covariance $\bSigma$.

\begin{assumption}[sub-Gaussian]\label{assumption:2}
	\textit{We assume that given $y=k$, the transformed variable $\bz^k=\bSigma^{-1/2}(\bx-\bmu_k)$ follows a sub-Gaussian distribution with mean $\zero$ and covariance $\bI_d$. In particular, there exists some $\sigma^2>0$ such that the following inequality holds,
		\$
		\EE[e^{\lambda \langle\bv,\bz^k\rangle}\mid y=k]\leq e^{\frac{\lambda^2\sigma^2}{2}}, \forall \lambda\in\RR, \bv\in\RR^{d},\norm{\bv}_2=1.
		\$}
\end{assumption}

Moreover, we assume a specific restricted eigenvalue (RE) condition on the population covariance matrix $\bSigma$, which is a common assumption in high-dimensional statistics \citep{Cai11,Lee17,Tian17,Wang17}. First, given a subset $\cA\subset[d]$, we define the  cone $\cC(\cA,3)\subset\RR^d$ as
\$
\cC(\cA,3)=\cbr{\bzeta\in\RR^d\mid\norm{\bzeta_{\cA^c}}_1\leq 3\norm{\bzeta_{\cA}}_1},
\$
where $\bzeta_{\cA}\in\RR^d$ is  $\bzeta$ with nonzeros restricted on the set $\cA$, i.e., $\bzeta_{\cA,i}=\bzeta_i$, if $i\in\cA$, $\bzeta_{\cA,i}=0$ otherwise; and $\bzeta_{\cA^{c}}$ is analogously defined. Then we assume $\bSigma$ satisfies the $RE(\kappa_0,s)$ condition for some constant $\kappa_0>0$ and positive integer $s$, which is defined below.

\begin{assumption}[Restricted Eigenvalue]\label{assumption:3}
	\textit{We assume that the covariance matrix $\bSigma$ satisfies the RE($\kappa_0,s$) condition, i.e., there exists some constant $\kappa_0>0$ such that
		\$
		\bzeta^{\top}\bSigma\bzeta\geq\kappa_0\norm{\bzeta}_2^2,\  \forall\bzeta\in \cC(\cA,3),\ |\cA|\leq s.
		\$}
\end{assumption}

Assumption \ref{assumption:2} and \ref{assumption:3} allow us to establish concentration inequalities for $\hat\bSigma$ and $\hat\bU$ as well as an RE condition on the sample covariance $\hat\bSigma^1$, which are stated in Lemmas \ref{lma:a1} and \ref{lma:a4}. We are ready to present our main theoretical results. We need to define the shifted loss function $\tilde\cL_1(\bW,\hat\bW_t)$ as follows,
\#
\tilde \cL_1(\bW, \hat \bW_t)=&\cL_1(\bW)+\langle\frac{1}{M}\sum_{m=1}^M\nabla\cL_m(\hat \bW_t)\notag\\
&-\nabla\cL_1(\hat \bW_t),\bW\rangle.\label{equ:4.1}
\#
We establish in Lemma \ref{lma:1} an upper bound on the $\ell_{\infty,\infty}$ norm  $\opnorm{\nabla\tilde\cL_1(\bW^*,\hat\bW_t)}{\infty,\infty}$ of the gradient of the shifted loss function at $\bW^*$, which suggests an appropriate choice of the tuning parameter $\lambda_{t+1}$.

\begin{lemma}\label{lma:1}
	\textit{Suppose Assumptions \ref{assumption:1} and \ref{assumption:2} hold and $\sup_{j}\bSigma_{jj}\leq c_0$ for some constant $c_0$, then there exist some universal contants $c_1$ and $c_2$ such that
		\$
		\opnorm{\nabla\tilde \cL_1(\bW^*, \hat \bW_t)}{\infty,\infty} \leq a_n+b_n\opnorm{\hat \bW_t-\bW^*}{1,1},
		\$
		holds for all $t$ with probability at least $1-2\delta$, where
		\#
		a_n&=c_1\sqrt{\frac{\log(dM/\delta)}{N}}\opnorm{\bW^*}{1,1}\notag\\
		&\quad+c_2\sqrt{\frac{K\log(dMK/\delta)}{N}},\label{equ:4.2}\\
		b_n&=2c_1\sqrt{\frac{\log(dM/\delta)}{n}}.\label{equ:4.3}
		\#}
\end{lemma}

The $\ell_{1,1}$ regularization parameter $\lambda_{t+1}$ is set as twice the upper bound in Lemma \ref{lma:1}, i.e.,
\#\label{equ:lambda} 
\lambda_{t+1}=2a_n+2b_n\opnorm{\hat \bW_t-\bW^*}{1,1}.
\#
Although the right hand side of \eqref{equ:lambda} involves the oracle matrix $\bW^*$ and thus unknown, it however suggests how to tune the parameter in practice. For example, the parameter $\lambda_{t+1}$ should decrease as the estimation error $\opnorm{\hat \bW_t-\bW^*}{1,1}$ goes down during iterations. Next, with this specific choice of parameter $\lambda_{t+1}$, we establish in Theorem \ref{theorem:2} a recursive estimation error bound connecting the estimation error $\opnorm{\hat\bW_{t+1}-\bW^*}{}$ to that of the previous iteration $\opnorm{\hat\bW_t-\bW^*}{}$, where the norm is the $\ell_{1,1}$ or $\ell_{2,2}$ norm.

\begin{theorem}\label{theorem:2}
	\textit{Suppose that Assumptions \ref{assumption:1}, \ref{assumption:2}, and \ref{assumption:3} hold, $\sup_j\bSigma_{jj}\leq c_0$ for some constant $c_0$, and the regularization parameter $\lambda_{t+1}$ is set as in \eqref{equ:lambda}, then there exist some universal constants $c_1,c_2,c_3$ such that for $n\geq c_3s\log(d/\delta)$, the recursive estimation error bounds
		\$
		\opnorm{\hat \bW_{t+1}-\bW^*}{2,2} & \leq \frac{48a_n\sqrt{s}}{\kappa_0} \\
		&+ \frac{48b_n\sqrt{s}}{\kappa_0}\opnorm{\hat\bW_t-\bW^*}{1,1},\\
		\opnorm{\hat \bW_{t+1}-\bW^*}{1,1} & \leq \frac{192a_ns}{\kappa_0} \\
		&+ \frac{192b_ns}{\kappa_0}\opnorm{\hat\bW_t-\bW^*}{1,1},
		\$
		hold for all $t$ with probability at least $1-3\delta$, where $a_n$ and $b_n$ are given by \eqref{equ:4.2} and \eqref{equ:4.3},  respectively, and $s=|\cS|$ is the cardinality of the support $\cS$ of $\bW^*$.}
\end{theorem}
\begin{figure*}[h]
	\centering
	\includegraphics[width = \textwidth,height = 0.35\textheight]{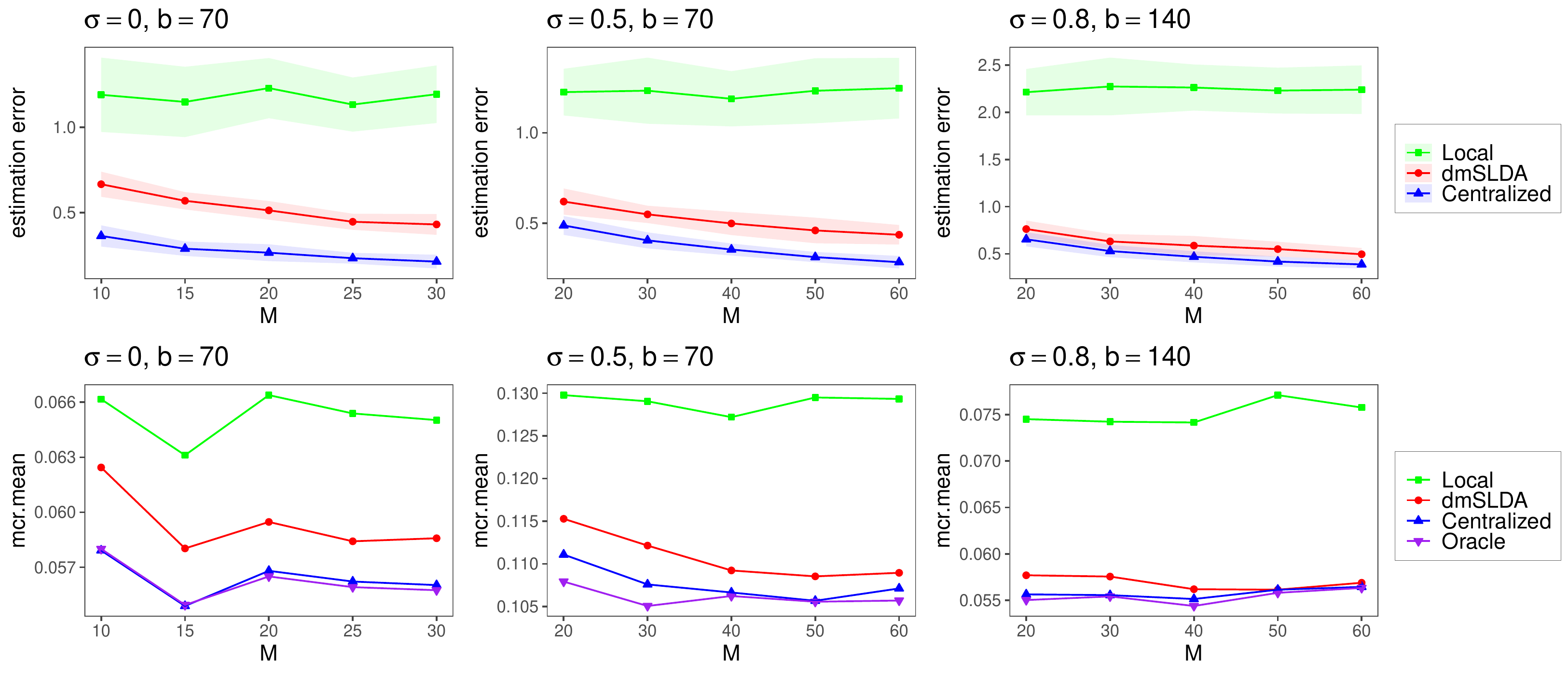}
	\caption{The top plots show the mean and the standard deviation of $\ell_{2,2}$ estimation error for dmSLDA, local SLDA, and centralized SLDA over 40 repetitions for different $\sigma,b$ and $M$. The bottom plots show the averaged MCR on the testing data for all four methods over 40 repetitions.}
	\label{fig:1}
\end{figure*}

Theorem \ref{theorem:2} upper bounds the $(t+1)$th estimation error $\opnorm{\hat\bW_{t+1}-\bW^*}{1,1}$ by a linear function of the $t$th estimation error $\opnorm{\hat\bW_{t}-\bW^*}{1,1}$. Thus by applying Theorem \ref{theorem:2} iteratively, we can bound $\opnorm{\hat\bW_{t+1}-\bW^*}{}$ in terms of the size of the local  estimation error $\opnorm{\hat\bW_0-\bW^*}{}$.

\begin{theorem}\label{theorem:3}
	\textit{Assume that the conditions in Theorem \ref{theorem:2} hold and the regularization parameter $\lambda_{t+1}$ is set as in \eqref{equ:lambda}, then there exist some universal constants $c_1,c_2,c_3$ such that for $n\geq c_3s\log(d/\delta)$, the estimation error bounds
		\$
		\opnorm{\hat \bW_{t+1}-\bW^*}{1,1}&\leq \frac{\frak a(1-\frak b^{t+1})}{1-\frak b}\\&+\frak b^{t+1}\opnorm{\hat \bW_{0}-\bW^*}{1,1},\\
		\opnorm{\hat \bW_{t+1}-\bW^*}{2,2}&\leq\frac{\frak a(1-\frak b^{t+1})}{4\sqrt{s}(1-\frak b)}\\&+\frac{\frak b^{t+1}}{4\sqrt{s}}\opnorm{\hat \bW_{0}-\bW^*}{1,1},
		\$
		hold with probablity at least $1-3\delta$, where $\mathfrak a=\frac{192a_ns}{\kappa_0}$ and  $\mathfrak b=\frac{192b_ns}{\kappa_0}$.}
\end{theorem}

Theorem \ref{theorem:3} provides theoretical guarantees that dmSLDA performs as well as the centralized SLDA in terms of the estimation accuracy. To see that, we observe that the coefficient $\mathfrak b<1$ when the local sample size $n$ is sufficiently large and thus the terms involving $\mathfrak b^{t+1}$ are negligible after a few (logarithmic) iterations. In other words, the estimation error for sufficiently large $t$ satisfies 
\$
\opnorm{\hat\bW_{t+1}-\bW^*}{2,2} \lesssim \sqrt{\frac{s\log(dM)}{N}},
\$
which is the same as that of the centralized method.

\section{NUMERICAL STUDIES}\label{sec:5}

In this section, we conduct numerical studies to futher back up our methodology and theory. The R code is released \href{https://github.com/HengchaoChen/dmSLDA.git}{here}.

\subsection{dmSLDA}
In this subsection, we conduct numerical simulations to illustrate the theoretical results in Section \ref{sec:4}. Four methods are implemented and compared: the dmSLDA, the local SLDA ($\hat\bW_0$), the centralized SLDA ($\hat\bW$ based on $\hat\bSigma=\frac{1}{M}\sum_{m=1}^M\hat\bSigma^m$ and $\hat\bU=\frac{1}{M}\sum_{m=1}^M\hat\bU^m$), and  the oracle ($\bW^*$). For the dmSLDA, we set $T=3$ as the maximum rounds of communication. In the experiments, we fix the number of classes $K=3$, the dimension $d=400$, and the class-conditional means $\bmu_1=(-2,-2,-2,0,\ldots,0)^\top$, $\bmu_2=(0,0,0,2,2,2,0,\ldots,0)^\top$, and $\bmu_3=(0,\ldots,0)^\top$. The class-conditional covariance $\bSigma$ is set as $\bSigma_{ij}=\sigma^{|i-j|}$ for all $1\leq i\leq j\leq d$, where $\sigma\in\{0,0.5,0.8\}$ gives three different settings and $0^0=1$ by convention. We set the subsample size $b=70$ of each class on one machine when $\sigma\in\{0,0.5\}$ and set $b=140$ when $\sigma=0.8$.\footnote{When $\sigma$ increases, $\kappa_0$ decreases and thus we need a larger subsample size.} Let the number of machines $M$ vary from $20$ to $60$ by an increment of $10$. For each setting, i.e., $(\sigma,b)\in\{(0,70),(0.5,70),(0.8,140)\}$ and $M\in\{20,30,\ldots,60\}$, we repeat the following procedure 40 times. First, we generate the training data $\{\bx^m_{kj}\}_{j=1}^{b}$ of class $k$ on machine $m$ and the testing data $\{\bx_{kj}\}_{j=1}^{300}$ of class $k$ from $\cN(\bmu_k,\bSigma)$ independently. Once the data are generated, we estimate the transformation matrix $\bW^*$ based on the training data and then compute both the $\ell_{2,2}$ estimation error $\opnorm{\hat\bW-\bW^*}{2,2}$ and the misclassification rate (MCR) based on the testing data. The results are displayed in Figure \ref{fig:1}.

It can be seen from Figure \ref{fig:1} that in all settings, the $\ell_{2,2}$ estimation error of the dmSLDA is close to that of the centralized SLDA, which is much smaller than that of the local SLDA, as implied by Theorem \ref{theorem:3}. Moreover, the averaged MCR of the centralized SLDA is almost equal to that of the oracle SLDA, to which the averaged MCR of the dmSLDA is comparable. They all outperform the local SLDA in terms of the averaged MCR. The rationale behind this phenomenon is that a more accurate estimator of $\bW^*$ leads to a better linear discriminant classifier.

\subsection{Binary Classification}

In this subsection, we compare the dmSLDA method with other distributed classifiers in the literature. To the best of our knowledge, dmSLDA is the first distributed multicategorical classifier for high dimensional data. Thus, we restrict our experiments to the binary setting. Specifically, we compare dmSLDA with the Centralized SLDA (Cen.SLDA) and two other distributed binary classifiers, the divide-and-conquer SLDA (dSLDA) method proposed by \citet{Tian17} and the divide-and-conquer $\ell_{1}$ SVM (dSVM) method proposed by \citet{Lian18}. We adopt the following binary balanced setting: $K=2$, $d=200$, $b=200$, $\bSigma_{ij}=0.8^{|i-j|}$, $\bmu_1=(\textnormal{rep}(0.5,10),0,\ldots)$, $\bmu_2=-\bmu_1$ and $M\in\{5,10,15,20,30\}$. For each setting, we generate the training data $\{\bx_{kj}^m\}_{j=1}^b$ of class $k$ on machine $m$ and the testing data $\{\bx_{kj}\}_{j=1}^{300}$ of class $k$ from $\cN(\bmu_k,\bSigma)$ independently and then implement all four methods. The whole procedure is repeated 40 times and the averaged misclassification rates (MCR) are reported in Figure \ref{fig:2}. It turns out that dmSLDA, dSLDA and Cen.SLDA are comparable and all these three methods outperform dSVM.
\begin{figure}[h]
	\centering
	\includegraphics[width = 0.4\textwidth,height = 0.175\textheight]{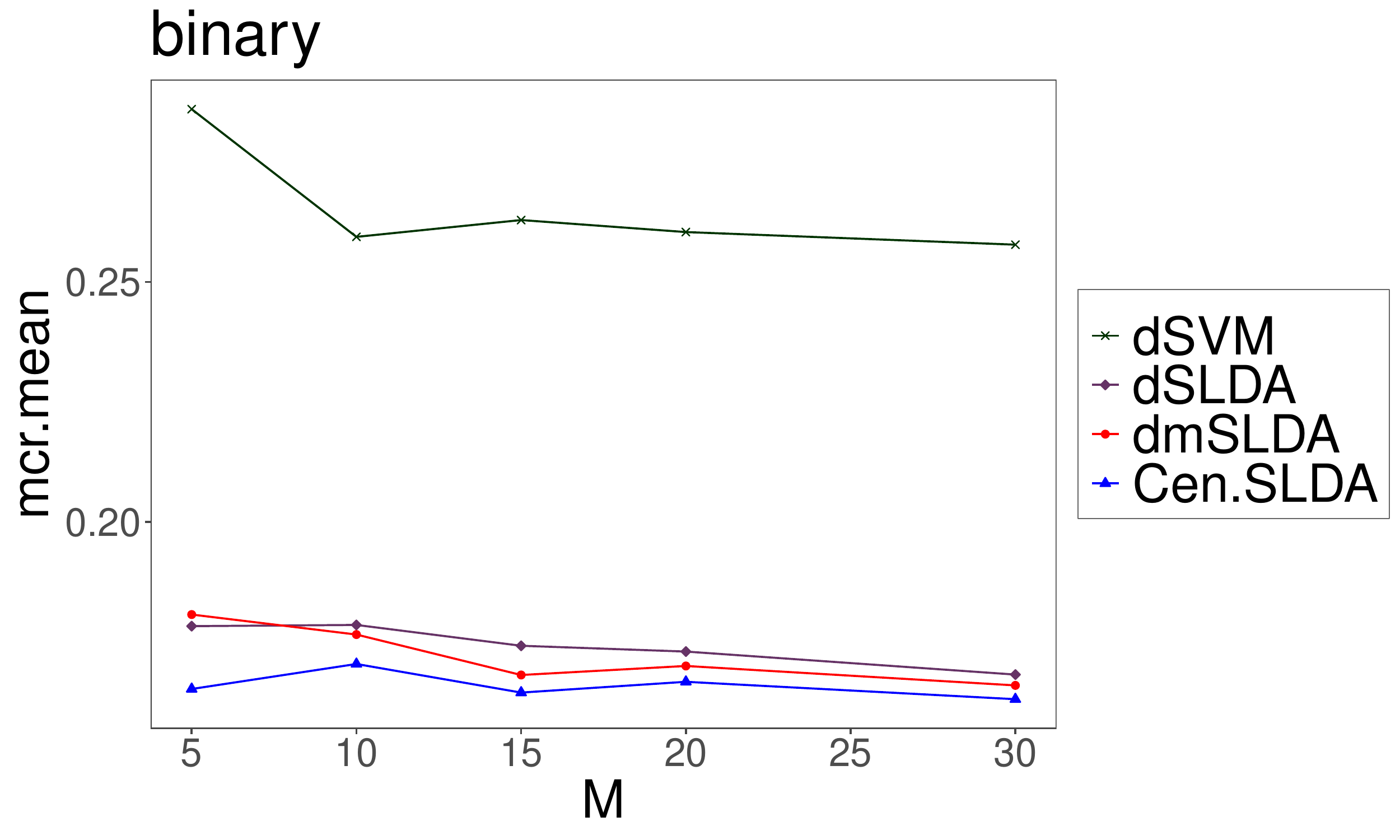}
	\caption{Comparsion between dSVM, dSLDA, dmSLDA, and Cen.SLDA. The averaged misclassification rate is computed over 40 repetitions.}
	\label{fig:2}
\end{figure}
\section{DISCUSSION}\label{sec:6}
In this paper, we propose an $\ell_{1,1}$ regularized multicategory SLDA method and extend it to the distributed setting when data are stored across multiple sites. We establish statistical properties ensuring that the distributed sparse multicategory LDA performs as good as the centralized sparse LDA after a few rounds of communications. Here we assume the data points follow sub-Gaussian distributions. It is possible to extend the current framework to deal with heavy-tailed data or data with adversarial contamination. We leave this to future work. 

\bibliography{references}


\clearpage
\appendix

\thispagestyle{empty}

\onecolumn \makesupplementtitle
\section{PROOF OF MAIN RESULTS}
\subsection{Proof of Proposition \ref{proposition:1}}
\begin{proof}[Proof of Proposition \ref{proposition:1}]
	Since $\bV,\bW\in\RR^{d\times q}$ span the same subspace, there exists an invertible matrix $\bR\in\RR^{q\times q}$ such that $\bV=\bW\bR$. The remaining proof follows the same argument of Proposition 4 in \citet{Gaynanova16}, while we note that the generalization of the orthogonal matrix $\bR$ to the general invertible matrix $\bR$ does not affect the proof.
\end{proof}
\subsection{Proof of Proposition \ref{proposition:2}}
\begin{proof}[Proof of Proposition \ref{proposition:2}]
	Since $\bSigma$ is invertible, the optimization problem \eqref{equ:2.3} has a closed form solution $\bW^*=\bSigma^{-1}\bU$. The core of the proof lies in the decomposition of the between-class covariance matrix $\bB=\bU\bP\bU^\top$ for some symmetric positive definite matrix $\bP\in\RR^{(K-1)\times (K-1)}$. By definition,
	\$
	\bB&=\sum_{k=1}^K\pi_k(\bmu_k-\bar \bmu)(\bmu_k-\bar\bmu)^\top\\
	&=\sum_{k=1}^{K-1}\pi_k(\bmu_k-\bar\bmu)(\bmu_k-\bar\bmu)^\top+\frac{1}{\pi_K}(\sum_{k=1}^{K-1}\pi_k(\bmu_k-\bar\bmu))(\sum_{k=1}^{K-1}\pi_k(\bmu_k-\bar\bmu))^\top,
	\$
	where the second equality holds since $\pi_K(\bmu_K-\bar\bmu)=-\sum_{k=1}^{K-1}\pi_k(\bmu_k-\bar\bmu))$. Then we can write $\bB=\bU\bP\bU^\top$, with $\bP\in\RR^{(K-1)\times (K-1)}$ given by
	\$
	\bP&=\textnormal{diag}(\pi_1,\ldots,\pi_{K-1})+\frac{1}{\pi_K}(\pi_1,\ldots,\pi_{K-1})^\top(\pi_1,\ldots,\pi_{K-1}).
	\$
	Note that $\bP$ is a symmetric positive definite matrix. Then we study the relationship between $\bW^*$ and discriminant vectors $\{\bv_k^*\}_{k=1}^q$ solved by Fisher's reduced rank LDA problem \eqref{equ:2.1}, where $q=\textnormal{rank}(\bB)$. Let $\bA=\bSigma^{-1/2}\bU\bP\bU^\top\bSigma^{-1/2}$, then the Fisher's reduced rank LDA problem can be transformed into a generalized eigenvalue problem
	\#
	&\tilde\bv_k^*=\argmax_{\tilde\bv_k\in\RR^d}\frac{\tilde\bv_k^{\top}\bA\tilde\bv_k}{\tilde\bv_k^{\top}\tilde\bv_k},\label{equ:a1}\\
	&\textnormal{s.t.}\quad \tilde\bv_k^{\top}\tilde\bv_j=0,\quad \forall 1\leq j\leq k\leq q,\notag
	\#
	where $\tilde\bv_k^*=\bSigma^{1/2}\bv_k^*$. Since $\textnormal{rank}(\bA)=\textnormal{rank}(\bB)=q$, the matrix $\bA$ can be orthogonally diagonalized as
	\$
	\bA=\bQ\textnormal{diag}(\lambda_1,\ldots,\lambda_{q},0,\ldots,0)\bQ^{\top},
	\$
	where $\bQ=(\ba_1,\ldots,\ba_d)\in\RR^{d\times d}$ is an orthogonal matrix and $\lambda_1\geq\cdots\geq\lambda_q>0$. Then the solution to the problem \eqref{equ:a1} is $(\tilde\bv^*_1,\ldots,\tilde\bv^*_{q})=(\ba_1,\ldots,\ba_{q})$ and thus the solution to the Fisher's reduced rank problem \eqref{equ:2.1} is $\bV_q=(\bv^*_1,\ldots,\bv^*_{q})=\bSigma^{-1/2}(\ba_1,\ldots,\ba_{q})$. As a result, we can express the linear subspace $\cV_q\coloneqq\textnormal{span}(\bV_q)$ as
	\$
	\cV_{q}&=\textnormal{span}(\bSigma^{-1/2}(\ba_1,\ldots,\ba_{q}))\\
	&\overset{(\rm{i})}{=}\textnormal{span}(\bSigma^{-1/2}(\sqrt{\lambda_1}\ba_1,\ldots,\sqrt{\lambda_{q}}\ba_{q}))\\
	&\overset{(\rm{ii})}{=}\textnormal{span}(\bSigma^{-1/2}\bA)\\
	&\overset{(\rm{iii})}{=}\textnormal{span}(\bSigma^{-1/2}(\bSigma^{-1/2}\bU\bP^{1/2}))\\
	&\overset{(\rm{iv})}{=}\textnormal{span}(\bSigma^{-1}\bU)=\textnormal{span}(\bW^*),
	\$
	where $(\rm{i})$ holds since $\lambda_1\geq\cdots\geq\lambda_q>0$, $(\rm{ii})$ holds since $\textnormal{span}(\sqrt{\lambda_1}\ba_1,\ldots,\sqrt{\lambda_q}\ba_q)=\textnormal{span}(\bA)$, $(\rm{iii})$ holds since $\textnormal{span}(\bA)=\textnormal{span}(\bSigma^{-1/2}\bU\bP^{1/2})$, and $(\rm{iv})$ holds since $\bP^{1/2}$ is invertible.
\end{proof}

\subsection{Proof of Lemma \ref{lma:1}}
\begin{proof}[Proof of Lemma \ref{lma:1}]\label{proof:1}
	Taking derivative of the shifted loss function $\tilde\cL_1(\bW,\hat\bW_t)$ with respect to the first entity at $\bW^*$, we get
	\#
	\nabla\tilde\cL_1(\bW^*,\hat\bW_t)
	&=\nabla\cL_1(\bW^*)+\frac{1}{M}\sum_{m=1}^M\nabla\cL_m(\hat\bW_t)-\nabla\cL_1(\hat\bW_t)\notag\\
	&=\frac{1}{M}\sum_{m=1}^M\nabla\cL_m(\bW^*)+\nabla\cL_1(\bW^*)-\nabla\cL_1(\hat\bW_1)+\frac{1}{M}\sum_{m=1}^M(\nabla\cL_m(\hat\bW_t)-\nabla\cL_m(\bW^*)),\notag
	\#
	where the second equality is obtained by adding and subtracting a term $\frac{1}{M}\nabla\cL_m(\bW^*)$. By definition of $\cL_m(\cdot)$, we have
	\#
	\nabla\tilde\cL_1(\bW^*,\hat\bW_t)=\frac{1}{M}\sum_{m=1}^M\nabla\cL_m(\bW^*)+(\hat\bSigma^1-\bSigma)(\bW^*-\hat\bW_t)+\frac{1}{M}\sum_{m=1}^M(\hat\bSigma^m-\bSigma)(\hat\bW_t-\bW^*).\notag
	\#
	Thus, by the triangle inequality, we have
	\#
	\opnorm{\nabla\tilde\cL_1(\bW^*,\hat\bW_t)}{\infty,\infty}\leq&\opnorm{\frac{1}{M}\sum_{m=1}^M\nabla\cL_m(\bW^*)}{\infty,\infty}\notag\\
	&+\left(\opnorm{\hat\bSigma^1-\bSigma}{\infty,\infty}+\frac{1}{M}\sum_{m=1}^M\opnorm{\hat\bSigma^m-\bSigma}{\infty,\infty}\right)\opnorm{\hat\bW_t-\bW^*}{1,1}.\label{equ:a2}
	\#
	Then we proceed to finish the proof by establishing upper bounds on the $\ell_{\infty,\infty}$ norm of $\frac{1}{M}\sum_{m=1}^M\nabla\cL_m(\bW^*)$ and $\hat\bSigma^m-\bSigma$. Fisrt, we establish in Lemma \ref{lma:a1} upper bounds on the  $\ell_{\infty,\infty}$ norm of the error matrix $\hat\bSigma^m-\bSigma$. Then in Lemma \ref{lma:a2}, we give upper bounds on the $\ell_{\infty,\infty}$ norm of the gradient term $\frac{1}{M}\sum_{m=1}^M\nabla\cL_m(\bW^*)$. 
	
	\begin{lemma}\label{lma:a1}
		\textit{Suppose the assumptions in Lemma \ref{lma:1} hold, then there exist some universal constants $c_1$ and $c_2$ such that the following bounds
			\#
			\opnorm{\hat\bSigma^m-\bSigma}{\infty,\infty} & \leq c_1\sqrt{\frac{\log(dM/\delta)}{n}},\quad \forall m\in[M],\label{conclusion:a1}\\
			\opnorm{\frac{1}{M}\sum_{m=1}^M\hat\bSigma^m-\bSigma}{\infty,\infty} & \leq c_1\sqrt{\frac{\log(dM/\delta)}{N}},\label{conclusion:a2}\\
			\opnorm{\hat\bU^m-\bU}{\infty,\infty} & \leq c_2 \sqrt{\frac{K\log(dMK/\delta)}{n}},\quad \forall m\in[M],\label{conclusion:a3}\\
			\opnorm{\frac{1}{M}\sum_{m=1}^M\hat\bU^m-\bU}{\infty,\infty}& \leq c_2\sqrt{\frac{K\log(dMK/\delta)}{N}},\label{conclusion:a4}
			\#
			hold with probability at least $1-2\delta$.}
	\end{lemma}

	\begin{lemma}\label{lma:a2}
		\textit{Assume that the covariance matrix $\Sigma$ is nonsingular and the upper bounds \eqref{conclusion:a1}, \eqref{conclusion:a2}, \eqref{conclusion:a3}, and \eqref{conclusion:a4} in Lemma \ref{lma:a1} hold, then we have for all $m\in[M]$,
			\#
			\opnorm{\nabla\cL_m(\bW^*)}{\infty,\infty}& \leq c_1\sqrt{\frac{\log(dM/\delta)}{n}}\opnorm{\bW^*}{1,1}+c_2\sqrt{\frac{K\log(dMK/\delta)}{n}},\label{conclusion:a5}\\
			\opnorm{\frac{1}{M}\sum_{m=1}^M\nabla\cL_m(\bW^*)}{\infty,\infty} & \leq c_1\sqrt{\frac{\log(dM/\delta)}{N}}\opnorm{\bW^*}{1,1}+c_2\sqrt{\frac{K\log(dMK/\delta)}{N}},\label{conclusion:a6}
			\#
			where the constants $c_1,c_2$ are defined in Lemma \ref{lma:a1}.}
	\end{lemma}
	
	By Lemma \ref{lma:a1} and \ref{lma:a2}, we have with probability at least $1-2\delta$ that inequalities \eqref{conclusion:a1}, \eqref{conclusion:a2}, \eqref{conclusion:a3}, \eqref{conclusion:a4}, \eqref{conclusion:a5} and \eqref{conclusion:a6} hold for some universal constants $c_1$ and $c_2$. Substituting these inequalities into the inequality \eqref{equ:a2}, we have
	\$
	\opnorm{\nabla\tilde \cL_1(\bW^*, \hat \bW_t)}{\infty,\infty} & \leq c_1\sqrt{\frac{\log(dM/\delta)}{N}}\opnorm{\bW^*}{1,1}+c_2\sqrt{\frac{K\log(dMK/\delta)}{N}}\\
	&\quad + 2c_1\sqrt{\frac{\log(dM/\delta)}{n}}\opnorm{\bW^*-\hat \bW_t}{1,1},
	\$
	with probability at least $1-2\delta$, which concludes the proof.
\end{proof}

\subsection{Proof of Theorem \ref{theorem:2}}
Before giving a proof of Theorem \ref{theorem:2}, we first state two technical lemmas that will be used. Let us begin with the definition of a series of subsets in $\RR^{d\times (K-1)}$. For any integer $s$, we define the subset $\cH(s)\subset\RR^{d\times (K-1)}$ by
\$
\cH(s)=\cbr{\bDelta\in\RR^{d\times (K-1)}\big|\bDelta_{\cdot k}\in\cC(\cA,3),\textnormal{ for some }|\cA|\leq s,\forall 1\leq k\leq K-1},
\$
where $\bDelta_{\cdot k}$ is the $k$th column of $\bDelta$. Furthermore, given an index set $\cS\subset[(d,K-1)]$, we define the cone $\tilde\cC(\cS,3)$ in $\RR^{d\times (K-1)}$ analogous to $\cC(\cA,3)$ in $\RR^d$ by
\$
\tilde\cC(\cS,3)=\cbr{\bDelta\in\RR^{d\times (K-1)}\big|\opnorm{\bDelta_{\cS^c}}{1,1}\leq3 \opnorm{\bDelta_\cS}{1,1}},
\$
where $\bDelta_\cS\in\RR^{d\times (K-1)}$ is the restriction of $\bDelta$ on the set $\cS$ and $\bDelta_{\cS^c}$ is analogously defined. It is shown in Lemma \ref{lma:a3} that the $t$th error $\hat\bW_t-\bW^*$ is in the subset $\cH(s)\cap \tilde \cC(\cS,3)$ under certain conditions, where $s=|\cS|$ and $\cS=\supp(\bW^*)$ is the support of $\bW^*$. In addition, we establish in Lemma \ref{lma:a4} a strong restricted convexity property of $\cL_1(\cdot)$, which proves to be useful. 

\begin{lemma}\label{lma:a3}
	\textit{Assume that the upper bound in Lemma \ref{lma:1} holds, $\lambda_{t+1}$ is chosen by \eqref{equ:lambda}, and the cardinality $|\cS|=s$, where $\cS=\supp(\bW^*)$ is the support of $\bW^*$. Then for any $t$,  the $(t+1)$th error $\hat \bW_{t+1}-\bW^*$ is in the subset $\cH(s)\cap\tilde \cC(\cS,3)$.}
\end{lemma}

\begin{lemma}[Strong Restricted Convexity]\label{lma:a4}
	\textit{Suppose the conditions in Theorem \ref{theorem:2} hold, then there exists some constant $c_3$ such that for $n\geq c_3s\log(d/\delta)$, the empirical loss function $\cL_1$ satisfies the strong restricted convexity property, i.e.,
		\#
		\cL_1(\bW^*+\bDelta)-\cL_1(\bW^*)-\inner{\nabla\cL_1(\bW^*)}{\bDelta}\geq \frac{\kappa_0}{4}\opnorm{\bDelta}{2,2}^2,\ \forall \bDelta\in\cH(s),\label{equ:RE}
		\#
		with probability at least $1-\delta$.}
\end{lemma}

\begin{remark}
	\textit{By definition of $\cL_1(\cdot)$, the strong restricted convexity is reduced to the following inequality,
		\$
		\langle\bDelta,\hat\bSigma^1\bDelta\rangle\geq\frac{\kappa_0}{2}\opnorm{\bDelta}{2,2}^2,\ \forall \bDelta\in\cH(s).
		\$
		As shown in the proof, we prove this inequality by establishing the restricted eigenvalue $RE(\kappa_0/2,s)$ property on the empirical covariance matrix $\hat\bSigma^1$. When we are finishing the paper, we find that the restricted eigenvalue property has been established in \citet{Rudelson13}. They proved the RE property via a reduction principle while in this paper we prove the RE property directly by applying Proposition 1 in \citet{Pilanci15}.}
\end{remark}

\begin{proof}[Proof of Theorem \ref{theorem:2}]\label{thmproof:2}
	By definition of $\tilde\cL_1(\cdot,\cdot)$, we have
	\#
	&\tilde\cL_1(\hat \bW_{t+1},\hat \bW_t)-\tilde\cL_1( \bW^*,\hat \bW_t)\notag\\
	=&\cL_1(\hat \bW_{t+1})-\cL_1(\bW^*)+\inner{\frac{1}{M}\sum_{m=1}^M\nabla\cL_m(\hat \bW_t)-\nabla\cL_1(\hat \bW_t)}{\hat\bW_{t+1}-\bW^*}.\label{equ:3.5.1}
	\#
	Since Assumption \ref{assumption:1} and \ref{assumption:2} hold and $\sup\bSigma_{jj}\leq c_0$, the upper bound in Lemma \ref{lma:1} holds for some universal constants $c_1$ and $c_2$ with probability at least $1-2\delta$. Since $\lambda_{t+1}$ is chosen by \eqref{equ:lambda}, the $(t+1)$th error $\hat\bW_{t+1}-\bW^*\in \cH(s)\cap\tilde\cC(\cS,3)$ by Lemma \ref{lma:a3}, where $s=|\cS|$ and $\cS=\supp(\bW^*)$ is the support of $\bW^*$. In addition, since Assumption \ref{assumption:3} also holds, the empirical loss function $\cL_1$ satisfies the strong restricted convexity property \eqref{equ:RE} when $n\geq c_3s\log(d/\delta)$ for some universal constant $c_3$ with probability at least $1-\delta$ by Lemma \ref{lma:a4}. In the remaining part of analysis, we first assume the conclusions in Lemma \ref{lma:1}, \ref{lma:a3} and \ref{lma:a4} hold and we will go back to the high-probability language in the end. 
	
	Combining the fact that $\hat\bW_{t+1}-\bW^*\in\cH(s)$ and the strong restricted convexity property \eqref{equ:RE} of $\cL_1$, we obtain the following inequality,
	\#
	\cL_1(\hat \bW_{t+1})-\cL_1(\bW^*)\geq\inner{\nabla\cL_1(\bW^*)}{\hat \bW_{t+1}-\bW^*}+\frac{\kappa_0}{4}\opnorm{\hat \bW_{t+1}-\bW^*}{2,2}^2.\label{equ:3.5.2}
	\#
	Substitute this inequality into the equation \eqref{equ:3.5.1}, we get
	\#
	&\tilde\cL_1(\hat \bW_{t+1},\hat \bW_t)-\tilde\cL_1( \bW^*,\hat \bW_t)\notag\\
	\geq&\inner{\nabla\cL_1(\bW^*)+\frac{1}{M}\sum_{m=1}^M\nabla\cL_m(\hat \bW_t)-\nabla\cL_1(\hat \bW_t)}{\hat \bW_{t+1}-\bW^*}+\frac{\kappa_0}{4}\opnorm{\hat \bW_{t+1}-\bW^*}{2,2}^2\notag\\
	=&\inner{\nabla\tilde\cL_1(\bW^*,\hat \bW_t)}{\hat \bW_{t+1}-\bW^*}+\frac{\kappa_0}{4}\opnorm{\hat \bW_{t+1}-\bW^*}{2,2}^2,\label{equ:3.5.3}
	\#
	where the equality follows from the definition of $\nabla\tilde\cL_1(\bW^*,\hat\bW_t)$. By the optimality \eqref{equ:3.2} of $\hat \bW_{t+1}$, we can derive the following inequality,
	\#
	\tilde\cL_1(\hat \bW_{t+1},\hat \bW_t)-\tilde\cL_1( \bW^*,\hat \bW_t)+\lambda_{t+1}\opnorm{\hat \bW_{t+1}}{1,1}-\lambda_{t+1}\opnorm{\bW^*}{1,1}\leq 0.\label{equ:3.5.4}
	\#
	Combining the inequalities \eqref{equ:3.5.3} and \eqref{equ:3.5.4}, the upper bound in Lemma \ref{lma:1}, and the definition \eqref{equ:lambda} of $\lambda_{t+1}$, we get
	\#
	&\lambda_{t+1}\opnorm{\bW^*}{1,1}-\lambda_{t+1}\opnorm{\hat \bW_{t+1}}{1,1}\geq \tilde\cL_1(\hat \bW_{t+1},\hat \bW_t)-\tilde\cL_1( \bW^*,\hat \bW_t)\notag\\
	\geq&\inner{\nabla\tilde\cL_1(\bW^*,\hat \bW_t)}{\hat \bW_{t+1}-\bW^*}+\frac{\kappa_0}{4}\opnorm{\hat \bW_{t+1}-\bW^*}{2,2}^2\notag\\
	\geq&-\opnorm{\nabla\tilde\cL_1(\bW^*,\hat \bW_t)}{\infty,\infty}\opnorm{\hat \bW_{t+1}-\bW^*}{1,1}+\frac{\kappa_0}{4}\opnorm{\hat \bW_{t+1}-\bW^*}{2,2}^2\notag\\
	\geq&-\frac{\lambda_{t+1}}{2}\opnorm{\hat \bW_{t+1}-\bW^*}{1,1}+\frac{\kappa_0}{4}\opnorm{\hat \bW_{t+1}-\bW^*}{2,2}^2,\label{equ:3.5.5}
	\#
	where the third inequality follows from the H$\rm{\ddot{o}}$lder's inequality. 
	Combining this inequality with the fact that $\hat\bW_{t+1}-\bW^*\in\tilde\cC(\cS,3)$, we obtain that
	\#
	\frac{\kappa_0}{4}\opnorm{\hat \bW_{t+1}-\bW^*}{2,2}^2&\leq \lambda_{t+1}\opnorm{\bW^*}{1,1}-\lambda_{t+1}\opnorm{\hat \bW_{t+1}}{1,1}+\frac{\lambda_{t+1}}{2}\opnorm{\hat \bW_{t+1}-\bW^*}{1,1}\notag\\
	&\leq\frac{3\lambda_{t+1}}{2}\opnorm{\hat \bW_{t+1}-\bW^*}{1,1}\notag\\
	&\leq6\sqrt{s}\lambda_{t+1}\opnorm{\hat \bW_{t+1}-\bW^*}{2,2},\label{equ:3.5.6}
	\#
	where the last inequality holds since $\opnorm{\hat\bW_{t+1}-\bW^*}{1,1}\leq4\sqrt{s}\opnorm{\hat\bW_{t+1}-\bW^*}{2,2}$ when $\hat\bW_{t+1}-\bW^*\in\tilde{\cC}(\cS,3)$.
	Dividing $\opnorm{\hat \bW_{t+1}-\bW^*}{2,2}$ from both sides of the inequality \eqref{equ:3.5.6}, we have 
	\#
	\opnorm{\hat \bW_{t+1}-\bW^*}{2,2}\leq\frac{24\sqrt{s}\lambda_{t+1}}{\kappa_0}.\label{equ:3.5.7}
	\#
	Again by the fact that $\hat\bW_{t+1}-\bW^*\in\tilde\cC(\cS,3)$, we have
	\#
	\opnorm{\hat \bW_{t+1}-\bW^*}{1,1}\leq4\sqrt{s}\opnorm{\hat \bW_{t+1}-\bW^*}{2,2}\leq\frac{96s\lambda_{t+1}}{\kappa_0}.\label{equ:3.5.8}
	\#
	Recall the definition \eqref{equ:lambda} of $\lambda_{t+1}$, the inequality \eqref{equ:3.5.7} and \eqref{equ:3.5.8} are exactly what we want. Finally, let us conclude the proof by noting that this analysis holds for all $t$ and thus the inequalities \eqref{equ:3.5.7} and \eqref{equ:3.5.8} hold for all $t$ with probability $1-3\delta$ when $n\geq c_3s\log(d/\delta)$.
\end{proof}

\subsection{Proof of Theorem \ref{theorem:3}}
\begin{proof}[Proof of Theorem \ref{theorem:3}]\label{thmproof:3}
	Setting $\frak{a} = \frac{192sa_n}{\kappa_0}$ and $\frak b = \frac{192sb_n}{\kappa_0}$, we can rewrite Theorem \ref{theorem:2} in the following iteration form,
	\$
	\opnorm{\hat \bW_{t+1}-\bW^*}{1,1}\leq \frak a+\frak b\opnorm{\hat \bW_t-\bW^*}{1,1},
	\$
	Applying Theorem \ref{theorem:2} iteratively, we can obtain the following upper bound on $\opnorm{\hat\bW_{t+1}-\bW^*}{1,1}$,
	\#
	\opnorm{\hat \bW_{t+1}-\bW^*}{1,1}&\leq \frak a+\frak b\opnorm{\hat \bW_t-\bW^*}{1,1}\notag\\
	&\leq \frak a+\frak b\left(\frak a+\frak b\opnorm{\hat \bW_{t-1}-\bW^*}{1,1}\right)\notag\\
	&\leq\cdots\notag\\
	&\leq \frac{\frak a(1-\frak b^{t+1})}{1-\frak b}+\frak b^{t+1}\opnorm{\hat \bW_0-\bW^*}{1,1}.\label{equ:a}
	\#
	Then by Theorem \ref{theorem:2} again, we can upper bound $\opnorm{\hat\bW-\bW^*}{2,2}$ as follows,
	\$
	\opnorm{\hat \bW_{t+1}-\bW^*}{2,2}&\leq\frac{\frak a}{4\sqrt{s}}+\frac{\frak b}{4\sqrt{s}}\opnorm{\hat \bW_{t}-\bW^*}{1,1}\\
	&\leq\frac{\frak a(1-\frak b^{t+1})}{4\sqrt{s}(1-\frak b)}+\frac{\frak b^{t+1}}{4\sqrt{s}}\opnorm{\hat \bW_{0}-\bW^*}{1,1},
	\$
	where the second inequality follows by \eqref{equ:a}.
\end{proof}

\section{PROOF OF AUXILIARY LEMMAS}
\subsection{Proof of Lemma \ref{lma:a1}}

\begin{proof}[Proof of Lemma \ref{lma:a1}]\label{proof:a1}
	First, let us upper bound the $\ell_{\infty,\infty}$ norm of $\hat\bSigma^m-\bSigma$ for a fixed $m$. By definition of $\hat\bSigma^m$ and Assumption \ref{assumption:1}, we have 
	\$
	\hat\bSigma^m-\bSigma & = \frac{1}{Kb-K}\sum_{k=1}^K\sum_{y_i^m=k}\cbr{(\bx_i^m-\bmu_k)(\bx_i^m-\bmu_k)^{\top}-\bSigma}\notag\\
	& - \frac{1}{Kb-K}\sum_{k=1}^K\cbr{b(\bar \bx^{mk}-\bmu_k)(\bar\bx^{mk}-\bmu_k)^{\top}-\bSigma},
	\$
	where $\bar\bx^{mk}=\frac{1}{b}\sum_{y^m_i=k}\bx_i^m$. For convenience, we introduce the following transformed random vectors,
	\$
	\tilde\bz_i^m & =\bSigma^{-1/2}(\bx_i^m-\bmu_k),\quad \textnormal{if }y_i^m=k,\quad\forall i,\\
	\tilde\bz^{mk} & =\sqrt{b}\bSigma^{-1/2}(\bar\bx^{mk}-\bmu_k),\quad \forall k.
	\$ 
	Note that conditional on $\cbr{y_i^m}_{i=1}^n$, $\cbr{\tilde\bz_i^m}_{i=1}^n$ are i.i.d. zero mean $\sigma$-sub-Gaussian random vectors and $\cbr{\tilde\bz^{mk}}_{k=1}^K$ are i.i.d. zero mean $\sigma$-sub-Gaussian random vectors, though $\cbr{\tilde\bz_i^m}_{i=1}^n$ and $\cbr{\tilde\bz^{mk}}_{k=1}^K$ can be dependent. Using these transformed random vectors, we can rewrite the error $\hat\bSigma^m-\bSigma$ as
	\$
	\hat\bSigma^m-\bSigma = \frac{Kb}{Kb-K}\bA - \frac{K}{Kb-K}\bA',
	\$
	where
	\#
	\bA & = \bSigma^{1/2}\rbr{\frac{1}{n}\sum_{i=1}^n\tilde\bz_i^m\tilde\bz_i^{m\top}-\bI_d}\bSigma^{1/2}, \label{A}\\
	\bA' & = \bSigma^{1/2}\rbr{\frac{1}{K}\sum_{k=1}^K\tilde\bz^{mk}\tilde\bz^{mk\top}-\bI_d}\bSigma^{1/2}.\label{A'}
	\#
	By the triangle inequality, we have
	\#\label{equ:b1}
	\opnorm{\hat\bSigma^m-\bSigma}{\infty,\infty}\leq \frac{Kb}{Kb-K}\opnorm{\bA}{\infty,\infty} + \frac{K}{Kb-K}\opnorm{\bA'}{\infty,\infty}.
	\#
	In the remaining part of the proof, we will apply Proposition 1 in \citet{Pilanci15} to upper bound the $\ell_{\infty,\infty}$ norm of $\bA$ and $\bA'$ separately, which leads to an upper bound on the $\ell_{\infty,\infty}$ norm of $\hat\bSigma^m-\bSigma$ by \eqref{equ:b1}. For reader's convenience, we present this proposition below.
	
	\begin{proposition}\label{proposition:a1}
		\textit{Let $\cbr{\bs_i}_{i=1}^n\subset\RR^d$ be i.i.d. samples from a zero-mean $\sigma$-sub-Gaussian distribution with $\Cov(\bs_i)=\bI_d$. Then there exist some universal constants $c_1,c_2$ such that for any subset $\cY\subset\SSS^{d-1}$, we have with probability at least $1-e^{-\frac{c_2n\delta^2}{\sigma^4}}$,
			\$
			\sup_{\etab\in\cY}\abr{\etab^{\top}\rbr{\frac{\bS^{\top}\bS}{n}-\bI_d}\etab}\leq c_1\frac{\WW(\cY)}{\sqrt{n}}+\delta,
			\$
			where $\bS^\top=\rbr{\bs_1,\ldots,\bs_n}$ and $\WW(\cY)$ is the Gaussian width of the subset $\cY$.  Specifically, $\WW(\cY)$ is defined by
			\$
			\WW(\cY)=\EE[\sup_{\etab\in\cY}\abr{\inner{\bg}{\etab}}],
			\$
			where the expectation is taken on $\bg\in\RR^d$, which is a standard normal random vector.}
	\end{proposition}
	
	Let us first bound the $\ell_{\infty,\infty}$ norm of $\bA$ by applying Proposition \ref{proposition:a1} to $\cbr{\bs_i}_{i=1}^n=\cbr{\tilde\bz_i^m}_{i=1}^n$ and $\cY=\cY_1\subset\SSS^{d-1}$ defined by
	\$
	\cY_1=\cbr{\bbeta_j=\frac{\bSigma^{1/2}_{\cdot j}}{\norm{\bSigma^{1/2}_{\cdot j}}_2}}_{1\leq j\leq d}\bigcup\cbr{\bgamma_{jk}=\frac{\bbeta_j+\bbeta_k}{\norm{\bbeta_j+\bbeta_k}_2}}_{1\leq j<k\leq d},
	\$
	where $\bSigma^{1/2}_{\cdot j}$ is the $j$th column of the matrix $\bSigma^{1/2}$. Note that the subset $\cY_1$ is a finite set with at most $\frac{d(d+1)}{2}$ elements. Since $\cbr{\tilde\bz_i^m}_{i=1}^n$ are i.i.d. samples from a zero mean $\sigma$-sub-Gaussian distribution with $\Cov(\tilde\bz_i^m)=\bI_d$, we have with probability at least $1-e^{-\frac{c_2n\delta^2}{\sigma^4}}$,
	\#\label{equ:b2}
	\sup_{\etab\in\cY_1}\abr{\etab^\top\rbr{\frac{1}{n}\sum_{i=1}^n\tilde\bz_i^m\tilde\bz_i^{m\top}-\bI_d}\etab}\leq c_1\frac{\WW(\cY_1)}{\sqrt{n}}+\delta,
	\#
	for some universal constants $c_1$ and $c_2$. By definition of $\bA$, we have 
	\$
	\opnorm{\bA}{\infty,\infty} & = \sup_{1\leq j\leq k\leq d} |\bA_{jk}| = \sup_{1\leq j\leq k\leq d}\abr{\bSigma^{1/2^\top}_{\cdot j}\rbr{\frac{1}{n}\sum_{i=1}^n\tilde\bz_i^m\tilde\bz_i^{m\top}-\bI_d}\bSigma^{1/2}_{\cdot k}}\notag \\
	& = \sup_{1\leq j\leq k\leq d} \abr{\bbeta_j^{\top}\rbr{\frac{1}{n}\sum_{i=1}^n\tilde\bz_i^m\tilde\bz_i^{m\top}-\bI_d}\bbeta_k}\cdot (\bSigma_{jj}\bSigma_{kk})^{1/2}\notag \\
	& \overset{(\rm{i})}{\leq} c_0\cdot \sup_{1\leq j\leq k\leq d} \abr{\bbeta_j^{\top}\rbr{\frac{1}{n}\sum_{i=1}^n\tilde\bz_i^m\tilde\bz_i^{m\top}-\bI_d}\bbeta_k},
	\$
	where the inequality $(\rm{i})$ follows from the condition $\sup_{1\leq j\leq d}\bSigma_{jj}\leq c_0$. By the polarization identity, we have 
	\$
	& \bbeta_j^{\top}\rbr{\frac{1}{n}\sum_{i=1}^n\tilde\bz_i^m\tilde\bz_i^{m\top}-\bI_d}\bbeta_k\notag\\
	= & \frac{1}{2}(\bbeta_j+\bbeta_k)^{\top}\rbr{\frac{1}{n}\sum_{i=1}^n\tilde\bz_i^m\tilde\bz_i^{m\top}-\bI_d}(\bbeta_j+\bbeta_k)\notag\\
	- & \frac{1}{2}\bbeta_j^{\top}\rbr{\frac{1}{n}\sum_{i=1}^n\tilde\bz_i^m\tilde\bz_i^{m\top}-\bI_d}\bbeta_j-\frac{1}{2}\bbeta_k^{\top}\rbr{\frac{1}{n}\sum_{i=1}^n\tilde\bz_i^m\tilde\bz_i^{m\top}-\bI_d}\bbeta_k.
	\$
	Then by the triangle inequality, we have 
	\$
	\opnorm{\bA}{\infty,\infty} & \leq  c_0\cdot \sup_{1\leq j\leq k\leq d} \abr{\bbeta_j^{\top}\rbr{\frac{1}{n}\sum_{i=1}^n\tilde\bz_i^m\tilde\bz_i^{m\top}-\bI_d}\bbeta_k} \notag \\
	& \leq c_0 \cdot \frac{1}{2} \cdot \norm{\bbeta_j+\bbeta_k}_2^2 \cdot \sup_{1\leq j\leq k\leq d} \abr{\bgamma_{jk}^{\top}\rbr{\frac{1}{n}\sum_{i=1}^n\tilde\bz_i^m\tilde\bz_i^{m\top}-\bI_d}\bgamma_{jk}} \notag \\
	& + c_0 \cdot \sup_{1\leq j\leq d}\abr{\bbeta_j^\top\rbr{\frac{1}{n}\sum_{i=1}^n\tilde\bz_i^m\tilde\bz_i^{m\top}-\bI_d}\bbeta_j}\notag \\
	& \leq 3c_0\cdot  \sup_{\etab\in\cY_1}\abr{\etab^\top\rbr{\frac{1}{n}\sum_{i=1}^n\tilde\bz_i^m\tilde\bz_i^{m\top}-\bI_d}\etab},
	\$
	where the last inequality holds since $\norm{\bbeta_j+\bbeta_k}_2^2\leq 4$ and $\cbr{\bbeta_j,\bgamma_{jk}}_{1\leq j\leq k\leq d}\subset\cY_1$. Combining this inequality with the upper bound in \eqref{equ:b2}, we can upper bound the $\ell_{\infty,\infty}$ norm of $\bA$ in terms of the Gaussian width $\WW(\cY_1)$,
	\$
	\opnorm{\bA}{\infty,\infty} & \leq 3c_0\cdot  \sup_{\etab\in\cY_1}\abr{\etab^\top\rbr{\frac{1}{n}\sum_{i=1}^n\tilde\bz_i^m\tilde\bz_i^{m\top}-\bI_d}\etab}\notag\\
	& \leq 3c_0\rbr{c_1\frac{\WW(\cY_1)}{\sqrt{n}}+\delta},
	\$
	with probability at least $1-e^{-\frac{c_2n\delta^2}{\sigma^4}}$. Now we give an upper bound on the Gaussian width $\WW(\cY_1)$ using the maximal inequality \cite{Wainwright19}. Since the random variable $\inner{\bg}{\etab}$ is Gaussian with zero mean and variance 1 for any $\etab\in\cY_1\subset \SSS^{d-1}$, and the cardinality of $\cY_1$ is finite less than $\frac{d(d+1)}{2}$, we have
	\$
	\WW(\cY_1)=\EE[\sup_{\etab\in\cY_1}|\inner{\bg}{\etab}|]\leq \sqrt{2\log (d(d+1))}.
	\$
	Therefore, we have with probability at least $1-e^{-\frac{c_2n\delta^2}{\sigma^4}}$,
	\#\label{equ:b3}
	\opnorm{\bA}{\infty,\infty} \leq c_0\cdot (c_1\sqrt{\frac{\log (d)}{n}}+\delta),
	\#
	for some universal constants $c_1$ and $c_2$. Applying the same argument to $\bA'$, we obtain the following upper bound
	\#\label{equ:b4}
	\opnorm{\bA'}{\infty,\infty}\leq c_0\cdot (c_1\sqrt{\frac{\log (d)}{K}}+\delta),
	\#
	with probability at least $1-e^{-\frac{c_2K\delta^2}{\sigma^2}}$. Combining inequalities \eqref{equ:b1}, \eqref{equ:b3}, and \eqref{equ:b4}, and the observation that $\bA'$-term is negligible in $\hat\bSigma^m-\bSigma$ relative to $\bA$-term, we have with probability at least $1 - e^{-\frac{c_2n\delta^2}{\sigma^4}}$,
	\#\label{equ:b5}
	\opnorm{\hat\bSigma^m-\bSigma}{\infty,\infty}\leq\frac{b}{b-1}\opnorm{\bA}{\infty,\infty}+\frac{1}{b-1}\opnorm{\bA'}{\infty,\infty}\leq c_0\cdot c_1(\sqrt{\frac{\log (d)}{n}}+\delta),
	\#
	for some universal constants $c_1,c_2$. Furthermore, we can bound the $\ell_{\infty,\infty}$ norm of $\frac{1}{M}\sum_{m=1}^M\hat\bSigma^m-\bSigma$ in the same way. In specific, with probability at least $1-e^{-\frac{c_2n\delta^2}{\sigma^4}}$, we have 
	\#\label{equ:b6}
	\opnorm{\frac{1}{M}\sum_{m=1}^M\hat\bSigma^m-\bSigma}{\infty,\infty}\leq c_0\cdot c_1(\sqrt{\frac{\log (d)}{N}}+\frac{\delta}{\sqrt{M}}),
	\#
	for some constants $c_1,c_2$. Using the method of union bound for \eqref{equ:b5} for all $m\in[M]$ and \eqref{equ:b6} and the change of variable of $\delta$, we have 
	\#\label{equ:b7}
	\opnorm{\hat\bSigma^m-\bSigma}{\infty,\infty}\leq c_1\sqrt{\frac{\log(dM/\delta)}{n}},\quad \forall m,\notag\\
	\opnorm{\frac{1}{M}\sum_{m=1}^M\hat\bSigma^m-\bSigma}{\infty,\infty}\leq c_1\sqrt{\frac{\log(dM/\delta)}{N}},
	\#
	for some universal constant $c_1$, with probability at least $1-\delta$.
	
	Next, we derive an upper bound on the $\ell_{\infty,\infty}$ norm of the error $\hat\bU^m-\bU$ for a fixed $m$. By Assumption \ref{assumption:1}, we can rewrite $\bU_{\cdot k}$ in the following form 
	\$
	\bU_{\cdot k}=(\bmu_1,\ldots,\bmu_K)\bh^k,\quad \forall k\in\cbr{1,\ldots,K-1},
	\$
	where the vector $\bh^k\in\RR^{K}$ satisfies
	\$
	\bh^k_k=1-\frac{1}{K},\bh^k_j=-\frac{1}{K}, \forall j\neq k.
	\$
	Using $\bh^k$, we can express elements of $\hat\bU^m-\bU$ as
	\$
	[\hat\bU^m-\bU]_{jk}=\hat\bU_{jk}^m-\bU_{jk}=\tilde\bX^m_{j\cdot}\bh^k-\bU_{jk},
	\$
	where $\tilde\bX^m=(\bar \bx^{m1},\ldots,\bar\bx^{mK})\in\RR^{d\times K}$ is the standard unbiased estimator of $(\bmu_1,\ldots,\bmu_K)$. Note that the $(j,k)$th element of $\tilde\bX^m$ is a linear combination of coordinates of $\bx_i^m$ and thus $\tilde\bX_{jk}^m$ is sub-Gaussian with parameter
	\$
	\sigma^2_{jk}=\frac{\sigma^2\bSigma_{jj}}{b}\leq\tilde\sigma^2\coloneqq\frac{\sigma^2c_0}{b},
	\$
	where the inequality follows from the condition $\sup_{j}\bSigma_{jj}\leq c_0$. Furthermore, columns of $\tilde\bX^m$ are independent random vectors. Therefore, $[\hat\bU^m-\bU]_{jk}$ is zero-mean sub-Gaussian variable with parameter
	\$
	\norm{\bh^k}_2\cdot\tilde\sigma=\frac{\sqrt{K-1}}{K}\tilde\sigma\leq \tilde\sigma.
	\$
	Then by the maximal inequality, we have
	\#\label{equ:b8}
	\opnorm{\hat\bU^m-\bU}{\infty,\infty}=\sup_{1\leq j\leq d,1\leq k\leq K-1}\abr{\hat\bU_{jk}^m-\bU_{jk}}\leq \delta,
	\#
	with probability at least $1-2dKe^{-\frac{\delta^2b}{2c_0\sigma^2}}$. Applying the same argument to $\frac{1}{M}\sum_{m=1}^M\hat\bU^m-\bU$, we have
	\#\label{equ:b9}
	\opnorm{\frac{1}{M}\sum_{m=1}^M\hat\bU^m-\bU}{\infty,\infty}\leq \frac{\delta}{\sqrt{M}},
	\#
	with probability at least $1-2dKe^{-\frac{\delta^2b}{2c_0\sigma^2}}$. Using the method of union bound for \eqref{equ:b8} for all $m$ and \eqref{equ:b9} and the change of variable of $\delta$, we have with probability at least $1-\delta$,
	\#
	\opnorm{\hat\bU^m-\bU}{\infty,\infty} & \leq c_2 \sqrt{\frac{K\log(dMK/\delta)}{n}},\quad \forall m,\notag\\
	\opnorm{\frac{1}{M}\sum_{m=1}^M\hat\bU^m-\bU}{\infty,\infty}& \leq c_2\sqrt{\frac{K\log(dMK/\delta)}{N}},\label{equ:b10}
	\#
	for some universal constant $c_2$.
	
	Finally, we conclude our proof by using the method of union bound for \eqref{equ:b7} and \eqref{equ:b10}.
\end{proof}

\subsection{Proof of Lemma \ref{lma:a2}}

\begin{proof}[Proof of Lemma \ref{lma:a2}]\label{proof:a2}
	By definition of $\cL_m(\cdot)$ and $\bW^*=\bSigma^{-1}\bU$, the gradient of $\cL_m(\cdot)$ at the point $\bW^*$ is
	\$
	\nabla\cL_m(\bW^*)=\hat\bSigma^m\bW^*-\hat \bU^m=(\hat\bSigma^m-\bSigma)\bW^*+(\bU-\hat \bU^m).
	\$
	Then by the triangle inequality, we obtain the following inequalities,
	\$
	\opnorm{\nabla\cL_m(\bW^*)}{\infty,\infty} & \leq \opnorm{\hat\bSigma^m-\bSigma}{\infty,\infty}\opnorm{\bW^*}{1,1}+\opnorm{\hat \bU^m-\bU}{\infty,\infty},\\
	\opnorm{\frac{1}{M}\sum_{m=1}^M\nabla\cL_m(\bW^*)}{\infty,\infty} & \leq \opnorm{\frac{1}{M}\sum_{m=1}^M\hat\bSigma^m-\bSigma}{\infty,\infty}\opnorm{\bW^*}{1,1}+\opnorm{\frac{1}{M}\sum_{m=1}^M\hat \bU^m-\bU}{\infty,\infty}
	\$
	Substituting the upper bounds \eqref{conclusion:a1}, \eqref{conclusion:a2}, \eqref{conclusion:a3}, and \eqref{conclusion:a4} in Lemma \ref{lma:a1} into the previous two inequalities, we have for all $m\in[M]$,
	\$
	\opnorm{\nabla\cL_m(\bW^*)}{\infty,\infty}& \leq c_1\sqrt{\frac{\log(dM/\delta)}{n}}\opnorm{\bW^*}{1,1}+c_2\sqrt{\frac{K\log(dMK/\delta)}{n}},\\
	\opnorm{\frac{1}{M}\sum_{m=1}^M\nabla\cL_m(\bW^*)}{\infty,\infty} & \leq c_1\sqrt{\frac{\log(dM/\delta)}{N}}\opnorm{\bW^*}{1,1}+c_2\sqrt{\frac{K\log(dMK/\delta)}{N}},
	\$
	where the constants $c_1,c_2$ are defined in Lemma \ref{lma:a1}. 
\end{proof}

\subsection{Proof of Lemma \ref{lma:a3}}
\begin{proof}[Proof of Lemma \ref{lma:a3}]\label{proof:a3}
	We prove this lemma in two steps: first, we show that $\hat\bW_{t+1}-\bW^*\in\cH(s)$; second, we show that $\hat\bW_{t+1}-\bW^*\in\tilde \cC(\cS,3)$, where $\cS=\supp(\bW^*)$ and $s=|\cS|$. 
	
	\textit{Step 1: $\hat\bW_{t+1}-\bW^*\in\cH(s)$.} Before we give an proof, let us introduce some notations first. We use $\hat\bw^t_k$ and $\bw^*_k$ to denote the $k$th column of the matrix $\hat\bW_t$ and $\bW^*$, respectively. Then we use $\cS_k=\supp(\bw^*_k)$ to denote the support of the $k$th column of $\bW^*$. Note that the cardinality of $\cS_k$ is smaller than that of $\cS=\supp(\bW^*)$, which is $s$. 
	
	For each $k$ and $t$, we define the following transformed loss function
	\$
	\tilde\ell_{1k}(\bw,\hat\bw^t_k)=\ell_{1k}(\bw)+\inner{\frac{1}{M}\sum_{m=1}^M\nabla\ell_{mk}(\hat\bw^t_k)-\nabla\ell_{1k}(\hat\bw^t_k)}{\bw},
	\$
	where
	\$
	\ell_{mk}(\bw)=\frac{1}{2}\inner{\bw}{\hat\bSigma^m\bw}-\inner{\hat\bU^m_{\cdot k}}{\bw}.
	\$
	Since the loss function $\tilde\cL_1(\bW,\hat\bW_t)$ and the $\ell_{1,1}$ norm of $\bW$ can be decomposed in the form of 
	\$
	\tilde\cL_1(\bW,\hat\bW_t) & =\sum_{k=1}^{K-1}\tilde\ell_{1k}(\bw_k,\hat\bw^t_k),\\
	\opnorm{\bW}{1,1} & =\sum_{k=1}^{K-1}\norm{\bw_k}_1,
	\$
	where $\bw_k=\bW_{\cdot k}$ is the $k$th column of $\bW$, the optimization problem
	\$
	\hat\bW_{t+1}=\argmin_{\bW\in\RR^{d\times (K-1)}}\cbr{\tilde\cL_1(\bW,\hat\bW_t) + \lambda_{t+1}\opnorm{\bW}{1,1}}
	\$ 
	is equivalent to the following $K-1$ minimization problems
	\#
	\hat \bw^{t+1}_k=\argmin_{\bw_k\in\RR^d}\cbr{\tilde\ell_{1k}(\bw_k,\hat\bw_k^t)+\lambda_{t+1}\norm{\bw_k}_1},\quad \forall 1\leq k\leq K-1.\label{equ:optimization}
	\#
	Moreover, we have the following inequality
	\$
	\norm{\nabla\tilde\ell_{1k}(\bw^*_k,\hat\bw^t_k)}_{\infty}\leq \opnorm{\nabla\tilde \cL_1(\bW^*,\hat \bW_t)}{\infty,\infty}\leq \frac{\lambda_{t+1}}{2},
	\$
	since $\nabla\tilde\ell_{1k}(\bw^*_k,\hat\bw^t_k)$ is the $k$th column of $\nabla\tilde \cL_1(\bW^*,\hat \bW_t)$ and $\lambda_{t+1}$ is defined by \eqref{equ:lambda}.
	
	Now let us prove the lemma by anaylzing the $k$th column of $\hat\bW_{t+1}-\bW^*$ with $k$ and $t$ fixed. By the triangle inequality, we have
	\$
	\norm{\hat \bw^{t+1}_k}_1-\norm{ \bw^*_k}_1&=\norm{\bw^*_k+(\hat \bw^{t+1}_k-\bw^*_k)_{\cS_k^c}+(\hat \bw^{t+1}_k-\bw^*_k)_{\cS_k}}_1-\norm{\bw^*_k}_{1}\\
	&\geq\norm{\bw^*_k+(\hat \bw^{t+1}_k-\bw^*_k)_{\cS^c_k}}_{1}-\norm{(\hat \bw^{t+1}_k-\bw^*_k)_{\cS_k}}_{1}-\norm{\bw^*_k}_{1}\\
	&=\norm{(\hat \bw^{t+1}_k-\bw^*_k)_{\cS_k^c}}_{1}-\norm{(\hat \bw^{t+1}_k-\bw^*_k)_{\cS_k}}_{1},
	\$
	where $\cS_k=\supp(\bw^*_k)$. By the optimality \eqref{equ:optimization} of $\hat \bw^{t+1}_k$, we have
	\$
	\tilde \ell_{1k}(\hat \bw^{t+1}_k,\hat \bw^t_k)+\lambda_{t+1}\norm{\hat \bw^{t+1}_k}_{1}-\tilde \ell_{1k}(\bw^*_k,\hat \bw^t_k)-\lambda_{t+1}\norm{\bw^*_k}_{1}\leq 0.
	\$
	Combining the above two inequalities, we get
	\$
	\tilde \ell_{1k}(\hat \bw^{t+1}_k,\hat \bw^t_k)-\tilde \ell_{1k}(\bw^*_k,\hat \bw^t_k)+\lambda_{t+1}(\norm{(\hat \bw^{t+1}_k-\bw^*_k)_{\cS_k^c}}_{1}-\norm{(\hat \bw^{t+1}_k-\bw^*_k)_{\cS_k}}_{1})\leq0.
	\$
	By the convexity of $\tilde \ell_{1k}(\cdot,\hat \bw^t_k)$ we have
	\$
	\tilde \ell_{1k}(\hat \bw^{t+1}_k,\hat \bw^t_k)-\tilde \ell_{1k}(\bw^*_k,\hat \bw^t_k)\geq\inner{\nabla\tilde \ell_{1k}(\bw^*_k,\hat \bw_k^t)}{\hat \bw^{t+1}_k-\bw^*_k}.
	\$
	Therefore,
	\$
	0&\geq\inner{\nabla\tilde \ell_{1k}(\bw_k^*,\hat \bw_k^t)}{\hat \bw_k^{t+1}-\bw_k^*}+\lambda_{t+1}(\norm{(\hat \bw_k^{t+1}-\bw_k^*)_{\cS_k^c}}_{1}-\norm{(\hat \bw_k^{t+1}-\bw_k^*)_{\cS_k}}_{1})\\
	&\geq-\norm{\nabla\tilde \ell_{1k}(\bw^*_k,\hat \bw^t_k)}_{\infty}\norm{\hat \bw^{t+1}_k-\bw^*_k}_{1}+\lambda_{t+1}(\norm{(\hat \bw^{t+1}_k-\bw^*_k)_{\cS_k^c}}_{1}-\norm{(\hat \bw^{t+1}_k-\bw^*_k)_{\cS_k}}_{1}).
	\$
	Since $\frac{\lambda_{t+1}}{2}\geq\norm{\nabla\tilde \ell_{1k}(\bw^*_k,\hat \bw_k^t)}_{\infty}$, we have
	\$
	0&\geq-\frac{1}{2}\norm{\hat \bw_k^{t+1}-\bw_k^*}_{1}+\norm{(\hat \bw_k^{t+1}-\bw_k^*)_{\cS_k^c}}_{1}-\norm{(\hat \bw_k^{t+1}-\bw_k^*)_{\cS_k}}_{1}\\
	&=\frac{1}{2}\norm{(\hat \bw^{t+1}_k-\bw_k^*)_{\cS_k^c}}_{1}-\frac{3}{2}\norm{(\hat \bw_k^{t+1}-\bw^*_k)_{\cS_k}}_{1},
	\$
	which proves that $\hat\bW_{t+1}-\bW^*\in\cH(s)$ since $|\cS_k|\leq|\cS|=s$.

	\textit{Step 2: $\hat\bW_{t+1}-\bW^*\in\tilde \cC(\cS,3)$}. The proof of this step is analogous to that of the first step. By the triangle inequality, we have
	\$
	\opnorm{\hat \bW_{t+1}}{1,1}-\opnorm{ \bW^*}{1,1}&=\opnorm{\bW^*+(\hat \bW_{t+1}-\bW^*)_{\cS^c}+(\hat \bW_{t+1}-\bW^*)_{\cS}}{1,1}-\opnorm{\bW^*}{1,1}\\
	&\geq\opnorm{\bW^*+(\hat \bW_{t+1}-\bW^*)_{\cS^c}}{1,1}-\opnorm{(\hat \bW_{t+1}-\bW^*)_{\cS_k}}{1,1}-\opnorm{\bW^*}{1,1}\\
	&=\opnorm{(\hat \bW_{t+1}-\bW^*)_{\cS^c}}{1,1}-\opnorm{(\hat \bW_{t+1}-\bW^*)_{\cS}}{1,1},
	\$
	where $\cS=\supp(\bW^*)$. By the optimality of $\hat \bW_{t+1}$, we have
	\$
	\tilde \cL_1(\hat \bW_{t+1},\hat \bW_{t})+\lambda_{t+1}\opnorm{\hat \bW_{t+1}}{1,1}-\tilde  \cL_1(\bW^*,\hat \bW_{t})-\lambda_{t+1}\opnorm{\bW^*}{1,1}\leq 0.
	\$
	Combining the above two inequalities, we get
	\$
	\tilde  \cL_1(\hat \bW_{t+1},\hat \bW_{t})-\tilde \cL_1(\bW^*,\hat \bW_{t})+\lambda_{t+1}(\opnorm{(\hat \bW_{t+1}-\bW^*)_{\cS^c}}{1,1}-\opnorm{(\hat \bW_{t+1}-\bW^*)_{\cS}}{1,1})\leq0.
	\$
	By the convexity of $\tilde  \cL_1(\cdot,\hat \bW_{t})$ we have
	\$
	\tilde  \cL_1(\hat \bW_{t+1},\hat \bW_{t})-\tilde \cL_1(\bW^*,\hat \bW_{t})\geq\inner{\nabla\tilde \cL_1(\bW^*,\hat \bW_{t})}{\hat \bW_{t+1}-\bW^*}.
	\$
	Therefore,
	\$
	0&\geq\inner{\nabla\tilde \cL_1(\bW^*,\hat \bW_{t})}{\hat \bW_{t+1}-\bW^*}+\lambda_{t+1}(\opnorm{(\hat \bW_{t+1}-\bW^*)_{\cS^c}}{1,1}-\opnorm{(\hat \bW_{t+1}-\bW^*)_{\cS}}{1,1})\\
	&\geq-\opnorm{\nabla\tilde \cL_1(\bW^*,\hat \bW_{t})}{\infty,\infty}\opnorm{\hat \bW_{t+1}-\bW^*}{1,1}+\lambda_{t+1}(\opnorm{(\hat \bW_{t+1}-\bW^*)_{\cS^c}}{1,1}-\opnorm{(\hat \bW_{t+1}-\bW^*)_{\cS}}{1,1}).
	\$
	Since $\frac{\lambda_{t+1}}{2}\geq\opnorm{\nabla\tilde  \cL_1(\bW^*,\hat \bW_{t})}{\infty,\infty}$, we have
	\$
	0&\geq-\frac{1}{2}\opnorm{\hat \bW_{t+1}-\bW^*}{1,1}+\opnorm{(\hat \bW_{t+1}-\bW^*)_{\cS^c}}{1,1}-\opnorm{(\hat \bW_{t+1}-\bW^*)_{\cS}}{1,1}\\
	&=\frac{1}{2}\opnorm{(\hat \bW_{t+1}-\bW^*)_{\cS^c}}{1,1}-\frac{3}{2}\opnorm{(\hat \bW_{t+1}-\bW^*)_{\cS}}{1,1},
	\$
	which proves that $\hat\bW_{t+1}-\bW^*\in\tilde\cC(\cS,3)$.
\end{proof}
\subsection{Proof of Lemma \ref{lma:a4}}
\begin{proof}[Proof of Lemma \ref{lma:a4}]\label{proof:a4}
	By definition of $\cL_1(\cdot)$, we have for each $\bDelta\in\cH(s)$,
	\#\label{equ:c1}
	&\cL_1(\bW^*+\bDelta)-\cL_1(\bW^*)-\inner{\nabla\cL_1(\bW^*)}{\bDelta}=\frac{1}{2}\inner{\bDelta}{\hat\bSigma^1\bDelta}.\notag\\
	= &\frac{1}{2} \sum_{j=1}^{K-1}\bDelta_{\cdot j}^{\top}\bSigma\bDelta_{\cdot j}\cdot\cbr{1+\frac{\bDelta_{\cdot j}^{\top}(\hat\bSigma^1-\bSigma)\bDelta_{\cdot j}}{\bDelta_{\cdot j}^{\top}\bSigma\bDelta_{\cdot j}}},
	\#
	where $\bDelta_{\cdot j}$ is the $j$th column of $\bDelta$.
	In the remain part of proof, we first establish an upper bound on the following term
	\$
	\sup_{\bDelta\in\cH(s),1\leq j\leq K-1}\frac{\abr{\bDelta_{\cdot j}^{\top}(\hat\bSigma^1-\bSigma)\bDelta_{\cdot j}}}{\bDelta_{\cdot j}^{\top}\bSigma\bDelta_{\cdot j}},
	\$
	using Proposition \ref{proposition:a1} and the RE($\kappa_0,s$) condition of $\bSigma$, then we prove the strong restricted convexity of $\cL_1$ by the equation \eqref{equ:c1} and the RE($\kappa_0,s$) condition of $\bSigma$. 
	
	By definition of $\cH(s)$, we have
	\$
	\sup_{\bDelta\in\cH(s),1\leq j\leq K-1}\frac{\abr{\bDelta_{\cdot j}^{\top}(\hat\bSigma^1-\bSigma)\bDelta_{\cdot j}}}{\bDelta_{\cdot j}^{\top}\bSigma\bDelta_{\cdot j}}\leq\sup_{\bzeta\in\cC(\cA,3),|\cA|\leq s}\frac{\abr{\bzeta^{\top}(\hat\bSigma^1-\bSigma)\bzeta}}{\bzeta^{\top}\bSigma\bzeta}.
	\$
	Similar to the proof of Lemma \ref{lma:a1}, by definition of $\hat\bSigma^1$ and the triangle inequality, we have 
	\#\label{equ:triangle}
	\abr{\bzeta^{\top}(\hat\bSigma^1-\bSigma)\bzeta} \leq \frac{Kb}{Kb-K}\abr{\bzeta^{\top}\bA\bzeta}+\frac{K}{Kb-K}\abr{\bzeta^\top\bA'\bzeta},
	\#
	where $\bA$ and $\bA'$ are defined in \eqref{A} and \eqref{A'} for $m=1$. By Proposition \ref{proposition:a1}, for some constants $c_1$ and $c_2$, we have
	\#
	& \sup_{\bzeta\in\cC(\cA,3),|\cA|\leq s}\frac{\abr{\bzeta^\top\bA\bzeta}}{\bzeta^\top\bSigma\bzeta}=\sup_{\etab=\frac{\bSigma^{1/2}\bzeta}{\norm{\bSigma^{1/2}\bzeta}_2},\bzeta\in\cC(\cA,3),|\cA|\leq s}\abr{\etab^\top(\frac{1}{n}\sum_{i=1}^n\tilde\bz_i^1\tilde\bz_i^{1^\top}-\bI_d)\etab}\notag\\
	\leq & c_1\frac{\WW(\cY_2)}{\sqrt{n}}+\delta,\label{equ:c2}
	\#
	with probability at least $1-e^{-\frac{c_2n\delta^2}{\sigma^4}}$, where
	\$
	\cY_2=\cbr{\etab=\frac{\bSigma^{1/2}\bzeta}{\norm{\bSigma^{1/2}\bzeta}_2}:\bzeta\in\cC(\cA,3),|\cA|\leq s}\subset\SSS^{d-1}.
	\$
	Now let us upper bound the Gaussian width $\WW(\cY_2)$. The Gaussian width $\WW(\cY_2)$ is
	\$
	\WW(\cY_2)=\EE[\sup_{\etab\in\cY_2}|\inner{\bg}{\etab}|]=\EE[\sup_{\bzeta\in\cC(\cA,3,),|\cA|\leq s,\norm{\bSigma^{1/2}\bzeta}_2=1}|\inner{\bg}{\bSigma^{1/2}\bzeta}|].
	\$ 
	Then by H$\rm\ddot{o}$lder's inequality, we have
	\$
	\WW(\cY_2) & \leq \EE[\sup_{\bzeta\in\cC(\cA,3),|\cA|\leq s,\norm{\bSigma^{1/2}\bzeta}_2=1}\norm{\bSigma^{1/2}\bg}_\infty\norm{\bzeta}_1]\\
	& \overset{\rm{(i)}}{\leq} \EE[\sup_{\bzeta\in\cC(\cA,3),|\cA|\leq s,\norm{\bSigma^{1/2}\bzeta}_2=1}\norm{\bSigma^{1/2}\bg}_\infty\norm{\bzeta}_2]\cdot 4\sqrt{s}\\
	& \overset{\rm{(ii)}}{\leq}\EE[\norm{\bSigma^{1/2}\bg}_\infty]\cdot 4\sqrt{\frac{s}{\kappa_0}},
	\$
	where the inequality $\rm(i)$ follows from the property
	\$
	\norm{\bzeta}_1=\norm{\bzeta_\cA}_1+\norm{\bzeta_{\cA^c}}_1\leq 4\norm{\bzeta_\cA}_1\leq 4\sqrt{|\cA|}\norm{\bzeta_\cA}_2\leq 4\sqrt{s}\norm{\bzeta}_2,
	\$
	and the inequality $\rm(ii)$ holds due to the RE($\kappa_0,s$) condition of $\bSigma$ and $\norm{\bSigma^{1/2}\bzeta}_2=1$. Then using the maximal inequality and the assumption $\sup_j\bSigma_{jj}\leq c_0$, we can bound the Gaussian width $\WW(\cY_2)$ as follows
	\#\label{equ:c3}
	\WW(\cY_2)\leq \sqrt{\frac{32c_0s\log(2d)}{\kappa_0}}.
	\#
	
	Combining the upper bound \eqref{equ:c3} of the Gaussian width $\WW(\cY_2)$ with the bound \eqref{equ:c2}, we have with probability at least $1-e^{-\frac{c_2n\delta^2}{\sigma^4}}$, 
	\$
	\sup_{\bzeta\in\cC(\cA,3),|\cA|\leq s}\frac{\abr{\bzeta^\top\bA\bzeta}}{\bzeta^\top\bSigma\bzeta}\leq c_1\sqrt{\frac{s\log(d)}{n}}+\delta,
	\$
	for some constants $c_1,c_2$. Apply the same argument to $\bA'$ and we get with probability at least $1-e^{-\frac{c_2Kb\delta^2}{\sigma^4}}$,
	\$
	\sup_{\bzeta\in\cC(\cA,3),|\cA|\leq s}\frac{\abr{\bzeta^\top\bA'\bzeta}}{\bzeta^\top\bSigma\bzeta}\leq c_1\sqrt{\frac{s\log(d)}{K}}+\sqrt{b}\delta,
	\$
	for some contants $c_1$ and $c_2$. Combining these two upper bounds with the inequality \eqref{equ:triangle}, we have with probability at least $1-2e^{-\frac{c_2n\delta^2}{\sigma^4}}$,
	\$
	\sup_{\bzeta\in\cC(\cA,3),|\cA|\leq s}\frac{\abr{\bzeta^{\top}(\hat\bSigma^1-\bSigma)\bzeta}}{\bzeta^\top\bSigma\bzeta}\leq c_1\sqrt{\frac{s\log(d)}{n}}+\delta,
	\$
	for some contants $c_1$ and $c_2$. Note again that $\bA'$-term is negligible relative to $\bA$-term. By change of variable of $\delta$, we have with probability at least $1-\delta$,
	\$
	\sup_{\bzeta\in\cC(\cA,3),|\cA|\leq s}\frac{\abr{\bzeta^{\top}(\hat\bSigma^1-\bSigma)\bzeta}}{\bzeta^\top\bSigma\bzeta}\leq c_1\sqrt{\frac{s\log(d/\delta)}{n}}.
	\$
	Then there exists some constant $c_3$, if $n\geq c_3s\log(d/\delta)$, 
	\$
	&\abr{1+\frac{\bDelta_{\cdot j}^{\top}(\hat\bSigma^1-\bSigma)\bDelta_{\cdot j}}{\bDelta_{\cdot j}^{\top}\bSigma\bDelta_{\cdot j}}}\notag\\
	\geq & 1-\sup_{\bzeta\in\cC(\cA,3),|\cA|\leq s}\frac{\abr{\bzeta^{\top}(\hat\bSigma^1-\bSigma)\bzeta}}{\bzeta^\top\bSigma\bzeta}\notag\\
	\geq & 1/2,
	\$
	with probability at least $1-\delta$. 
	
	By the RE($\kappa_0,s$) condition of $\bSigma$ and the property that $\bDelta_{\cdot j}\in\cC(\cA,3)$ for some $|\cA|\leq s$, we have
	\$
	\bDelta_{\cdot j}^{\top}\bSigma\bDelta_{\cdot j}\geq \kappa_0\norm{\bDelta_{\cdot j}}_2^2.
	\$
	Thus we have with probability at least $1-\delta$,
	\$
	&\cL_1(\bW^*+\bDelta)-\cL_1(\bW^*)-\inner{\nabla\cL_1(\bW^*)}{\bDelta}\geq \frac{1}{4}\sum_{j=1}^{K-1}\bDelta_{\cdot j}^{\top}\bSigma\bDelta_{\cdot j}\notag\\
	\geq & \frac{\kappa_0}{4}\sum_{j=1}^{K-1}\norm{\bDelta_{\cdot j}}_2^2=\frac{\kappa_0}{4}\opnorm{\bDelta}{2,2}^2,
	\$
	if $n\geq c_3s\log(d/\delta)$ for some constant $c_3$.
\end{proof}

\end{document}